\documentclass[reqno]{amsart}            
\usepackage{graphicx,cite,color}
\usepackage[bookmarks,bookmarksnumbered,bookmarksopen,colorlinks,backref,linkcolor=blue,citecolor=red]{hyperref}%
\usepackage{xcolor}
\usepackage{dsfont}
\usepackage{mathrsfs}
\usepackage{amsmath, amssymb ,amsthm, amsfonts, amsgen}

\newtheorem{theorem}{Theorem}[section]
\newtheorem{definition}{Definition}[section]

\newtheorem{corollary}{Corollary}[section]
\newtheorem{remark}{Remark}[section]
\numberwithin{equation}{section}

\begin{document}

\title[Puiseux  asymptotic expansion for  transport problem in thin networks]
{Puiseux  asymptotic expansions for convection-dominated transport problems in thin graph-like networks:\\
 strong boundary interactions}
\author[Taras Mel'nyk \& Christian Rohde]{ Taras Mel'nyk$^{\natural, \flat}$ \ \& \ Christian Rohde$^\natural$}
\address{\hskip-12pt
$^\natural$ Institute of Applied Analysis and Numerical Simulation,
Faculty of Mathematics and Physics,  University of Stuttgart\\
Pfaffenwaldring 57,\ 70569 Stuttgart,  \ Germany
 \newline
$^\flat$ Department of Mathematical Physics, Faculty of Mathematics and Mechanics\\
Taras Shevchenko National University of Kyiv\\
Volodymyrska str. 64,\ 01601 Kyiv,  \ Ukraine\\
}
\email{Taras.Melnyk@mathematik.uni-stuttgart.de}

\address{\hskip-12pt  $^\natural$ Institute of Applied Analysis and Numerical Simulation,
Faculty of Mathematics and Physics, University of Stuttgart\\
Pfaffenwaldring 57,\ 70569 Stuttgart,  \ Germany
}
\email{christian.rohde@mathematik.uni-stuttgart.de }

\begin{abstract} \vskip-10pt
This article completes the study of the influence of the intensity parameter $\alpha$ in the boundary condition
$
\varepsilon  \partial_{\boldsymbol{\nu}_\varepsilon} u_\varepsilon  -  u_\varepsilon \, \overrightarrow{V_\varepsilon}\boldsymbol{\cdot}\boldsymbol{\nu}_\varepsilon   =  \varepsilon^{\alpha} \varphi_\varepsilon
$
given on the boundary of a  thin three-dimensional graph-like network consisting of thin cylinders that are  interconnected by small domains (nodes) with diameters of order  $\mathcal{O}(\varepsilon).$ Inside of the thin network a time-dependent  convection-diffusion equation with high P\'eclet number  of  order $\mathcal{O}(\varepsilon^{-1})$ is considered. The novelty of this article is the case of $\alpha <1,$ which indicates a strong intensity of physical processes on the boundary, described by the inhomogeneity $\varphi_\varepsilon$ (the cases $\alpha =1$ and $\alpha >1$ were previously studied by the same authors).

A complete Puiseux asymptotic expansion is constructed for the solution $u_\varepsilon$ as $\varepsilon \to 0,$ i.e., when the diffusion coefficients are eliminated and the thin network shrinks  into a graph. Furthermore, the corresponding uniform pointwise and energy estimates are proved, which provide an approximation of the solution with a given accuracy in terms  of the parameter $\varepsilon.$
\end{abstract}

\keywords{Asymptotic expansion, convection-diffusion problem, boundary interactions, thin graph-like network, hyperbolic limit model
\\
\hspace*{9pt} {\it MOS subject classification:} \   35K20,  35R02,  35B40, 35B25, 	35B45, 35K57, 35Q49
}

\maketitle
\tableofcontents
\section{Introduction}\label{Sect1}

In this paper we continue our study \cite{Mel-Roh_preprint-2022,Mel-Roh_preprint-2023} of parabolic  
 transport problems for some species' concentration  in graph-like networks. For  $\varepsilon$ being a small parameter, these consist of thin cylinders that are  interconnected by small domains (nodes)   of order  $\mathcal{O}(\varepsilon)$ in diameter, see Fig.~\ref{f1}. The parameter  $\varepsilon$ is not only characterising the geometry of the thin network but also the strength of certain physical processes. First,  we consider  a high P\'eclet number regime of  order $\mathcal{O}(\varepsilon^{-1})$, i.e. a high ratio of  convective to  diffusive transport rates. This  assumption   leads to  a re-scaled  diffusion operator in the  parabolic differential equation
\begin{equation}\label{intr.1}
  \partial_t u_\varepsilon -  \varepsilon\, \Delta_x u_\varepsilon + \mathrm{div}_x \big( \overrightarrow{V_\varepsilon} \, u_\varepsilon\big)
   =  0,
\end{equation}
governing the unknown species concentration $u_\varepsilon$. In \eqref{intr.1}, $\overrightarrow{V_\varepsilon}$ is
 a  given convective vector field with small transversal velocity components.
In order to study additionally the influence of  external influences via boundary interactions, in the paper \cite{Mel-Roh_preprint-2023}  an additional parameter $\alpha$ was introduced  in the  inhomgeneous  boundary conditions
$$
\varepsilon  \partial_{\boldsymbol{\nu}_\varepsilon} u_\varepsilon  -  u_\varepsilon \, \overrightarrow{V_\varepsilon}\boldsymbol{\cdot}\boldsymbol{\nu}_\varepsilon   =  \varepsilon^{\alpha} \varphi^{(i)}_\varepsilon, \quad i \in \{0, 1, \ldots, \mathcal{M}\},
$$
which hold both on the lateral surfaces of thin cylinders and on the boundaries of nodes forming the  thin network.
Here,  a given function $\varphi^{(i)}$ describes physical processes on the surface of the $i$th component of the network.

Three qualitatively different cases can be identified  for the  asymptotic behaviour (as $\varepsilon \to 0$)  of the solution~$u_\varepsilon$ depending on the value of  $\alpha$, namely
\begin{description}
  \item[$\alpha > 1$]  low intensity of boundary processes (if $\alpha \ge 2$ they can simply be ignored, if $\alpha \in (1, 2)$ their influence appears in the second-order terms of the asymptotics);
  \item[$\alpha =1$] moderate intensity, then the boundary processes directly affect the first terms ${\{{w}^{(i)}_0\}}_i$ of the asymptotics, which represent the solution of the limit problem consisting of first-order hyperbolic equations
   \begin{equation}\label{intr_1}
     \partial_t{w}^{(i)}_0(x_i,t) \, + \, \partial_{x_i}\big( v_i^{(i)}(x_i,t)\, w^{(i)}_0(x_i,t) \big)
  =  - \widehat{\varphi}^{(i)},
   \end{equation}
 each of which is defined on the $i$-th edge of the corresponding metric graph, to which the thin network shrinks
(here $-\widehat{\varphi}^{(i)}$ is the limit transformation of the boundary interaction $\varphi^{(i)}$ into the right-hand side of the differential equation);
    \item [$\alpha <1$] strong intensity (it was only noted in the conclusion of  \cite{Mel-Roh_preprint-2023} that the question of finding the asymptotics remains open and one should expect that the solution will be unbounded  as  $\varepsilon\to 0$).
\end{description}
In this paper we present a complete analysis of the latter, qualitatively different case whereas the first two ones have been analyzed in
\cite{Mel-Roh_preprint-2023}.
\begin{figure}[htbp]
\centering
\includegraphics[width=8cm]{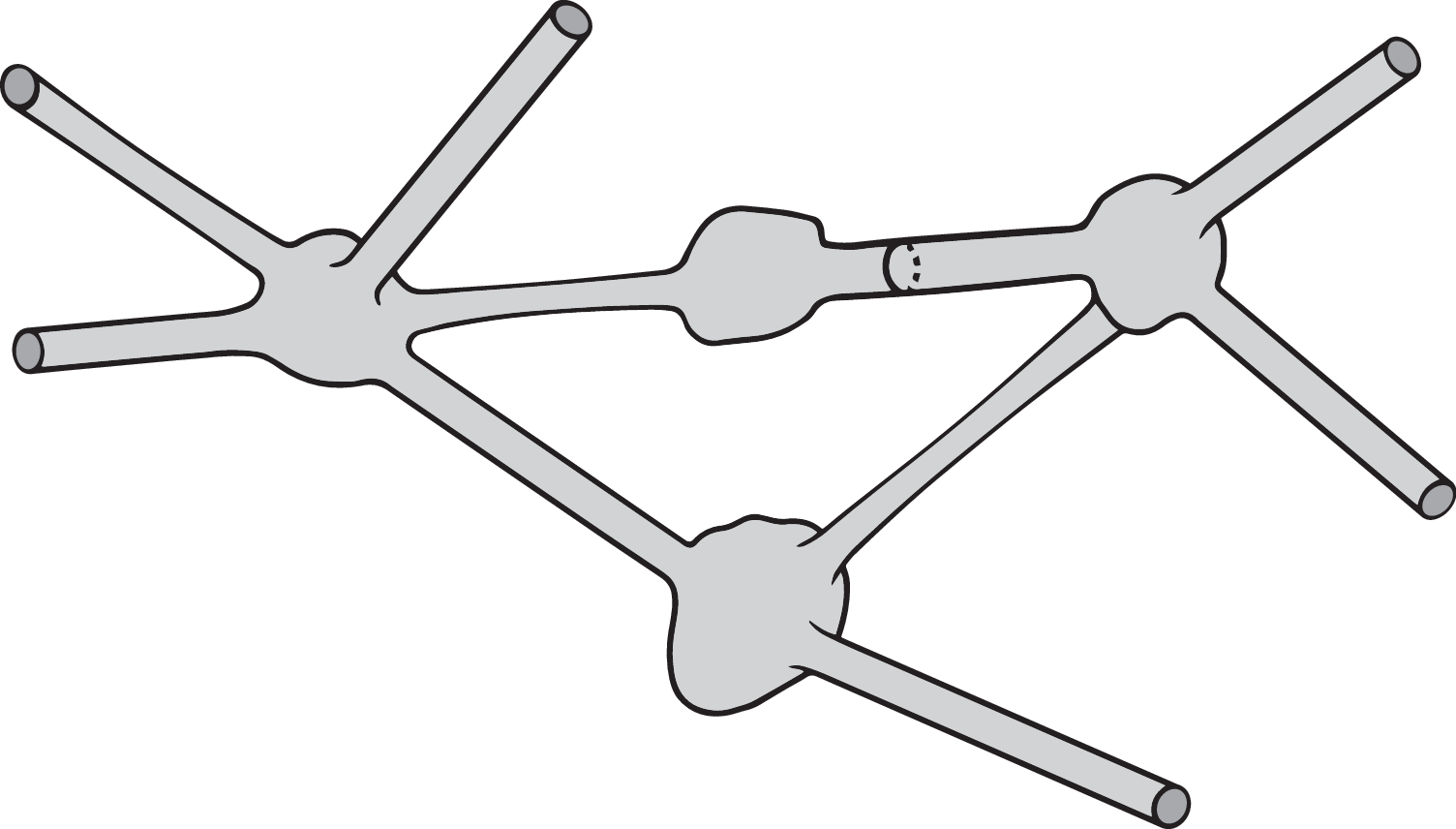}
\caption{Network of thin cylinders connected by  nodes of arbitrary geometry.}\label{f1}
\end{figure}

Studying the influence of inhomogeneous Neumann or Fourier boundary conditions on the processes in the entire domain is a very timely task, especially in regions with complex structures,  e.g., see  papers \cite{Ang_2020,Cio-Dam-Don-Gri-Zak2012,Con-Don_1988,Conca,Conca_1,Deu-Hoch_2004,Don-Gui-Oro_2018,Gom-Lob-Per-San_2018,Gom-Lob-Per-Pod-Sha_2018,
Hor-Jag_1991,Mel-Siv_2011} for  perforated domains,  \cite{CarPanSir,Mel-Kle_2019,Mel_IJB-2019} for thin domains,   \cite{Che-Fri-Pia_1999,Mel_1991,Mel-Pop_2009} for domains with rapidly oscillating boundaries,  \cite{Gau_1994,Gau-Mel_2019,Mel-Sad_2019} for thick junctions or domains with a highly oscillating boundary.
The problems considered in these works are mathematical models of various reaction-diffusion, adsorption and transport phenomena in hydrogeology, chemistry, biology, medicine, and thermal conductivity.

In the vast majority of these works, the intensity of the processes at the boundaries is not high, which  leads to the appearance of an additional term in the corresponding limit differential equation (it depends on the ratio between the total size of the surfaces and the intensity of the processes on these surfaces).
In \cite{Gau_1994} the author gave an example of a boundary-value problem with a large boundary interaction in a domain with a highly oscillating boundary and proved that the solution becomes unbounded as the oscillation parameter tends to zero.
 Additional assumptions that guarantee a priori estimates for the solution regardless of the small parameter can help to prove the convergence theorem even for large boundary interactions (see, e.g.,  \cite{Gau-Mel_2019}).
However, in the general case,  physical processes at the boundary of regions with a complex structure can provoke cardinal changes in the global phenomenon throughout the region.

The principal novelty of the present paper is the study of the impact of strong boundary interactions (the case $\alpha < 1$ and
without any additional assumptions guaranteeing a priori estimates) on the asymptotic behaviour (as $\varepsilon \to 0$) of the solution to the  parabolic convection-dominated transport problem in a thin graph-like network.

To do this, we use the asymptotic expansion method, which is a very powerful tool for studying various perturbed problems.
Usually the corresponding first term in the expansion gives the basic limiting behavior for the solution of interest; the subsequent terms allow to estimate the sensitivity of the model  to smaller scale  perturbations. Therefore, convergence results used in studies of perturbed models may not be acceptable in cases where the values of the perturbation parameters are not small enough.
Asymptotic expansions allow the influence of other features of the model to be taken into account through higher-order terms in the expansions, and improve the accuracy of the corresponding numerical procedures by combining them with asymptotic results.
It should be noted that in asymptotic analysis it is very important to guess the form of the asymptotic ansatz,  which is affected by various parameters of the problem (see. e.g., \cite{Mel-Siv_2011,M-AA-2021}).

In the present work we compose effective recurrent algorithms for constructing asymptotic Puiseux  expansions. Such  series are a generalization of power series that allow negative and fractional exponents of the indeterminate; in our case, these are real exponents of the parameter $\varepsilon$, which depend on the parameter~$\alpha$. We  show that the  principal part of the  Puiseux series
(it  consists of unbounded coefficients as $\varepsilon \to 0$;  see Definition~\ref{Puiseux}), which determines the main behaviour of the solution in the $i\!$th thin cylinder of the network, is  as follows
\begin{equation}\label{main_part}
\varepsilon^{\alpha -1} \, w_{\alpha -1}^{(i)} (x_i,t)  + \varepsilon^{\alpha} \, w_{\alpha}^{(i)} (x_i,t)
  +
  \sum\limits_{k=2}^{-\lfloor\alpha\rfloor}  \varepsilon^{\alpha +k -1} \, \Big(w_{\alpha +k -1}^{(i)} (x_i) + u_{\alpha +k -1}^{(i)} \big( x_i, \tfrac{\overline{x}_i}{\varepsilon}, t\big)\Big).
\end{equation}
Note that, for  $\alpha \in [0,1)$, there is only one term $\varepsilon^{\alpha -1} \, w_{\alpha -1}^{(i)}$ in the principal part. The coefficients ${\{w_{\alpha -1}^{(i)}\}}_i$ form the solution to the corresponding hyperbolic transport problem on the graph,  the right-hand sides  of which are the averaged characteristics of  the boundary interactions (see \eqref{limit_al-1}). The coefficients ${\{w_{\alpha}^{(i)}\}}_i$ represent  the solution to the hyperbolic transport problem \eqref{limit_al},  and they take  into account interactions on the node boundary, physical processes inside the node and geometric characteristics of the network  through the special gluing condition
 \begin{equation}\label{gluing}
 \sum_{i}  \mathrm{v}_i  \, h_i^2\, w_{\alpha}^{(i)} (0,t) = \boldsymbol{d_{\alpha}}(t)
 \end{equation}
  at the graph vertex, where $\boldsymbol{d_{\alpha}}(t)$ is determined in \eqref{const_d_alpfa}.

It is worth noting another interesting feature in the study of boundary-value problems with a predominance of convection in thin networks, namely that the corresponding first-order hyperbolic problems for the coefficients of the asymptotic expansion have only a Kirchhoff-type  transmission condition (similar to \eqref{gluing}) at the graph vertices; no other condition, commonly called the continuity condition.
This means that such solutions, generally speaking, are not continuous at the vertices, which is explained by the wave-particle duality of first-order hyperbolic differential equations (for more details see the conclusion section of our paper \cite{Mel-Roh_preprint-2022}).
To neutralize this duality and to ensure continuity of the approximation in the vicinities of nodes, a special node-layer part of the asymptotics is introduced whose coefficients are special solutions (with polynomial growth at infinity) to boundary value problems in an unbounded domain with different outlets at infinity (see  \eqref{N_p_alpha +k_prob}).  Writing down the solvability conditions for such problems, Kirchhoff-type  conditions are derived at the vertex of the graph for terms of the regular expansion.

For the sake of clarity of presentation and to focus more on the study of the effects of boundary processes, we consider a simple model of a thin graph-like network consisting of three thin cylinders located along the coordinate axes, respectively, and the Laplace equation for modelling diffusion processes. A more general diffusion operator $\varepsilon\, \mathrm{div}_x\big( \mathbb{D}(\frac{x}{\varepsilon}) \nabla_x u_\varepsilon\big)$  as well as thin junctions with curvilinear cylinders  were considered in our previous paper \cite{Mel-Roh_preprint-2022}, and general thin networks with different convective vector field dynamics   in \cite{Mel-Roh_preprint-2023}.

The  paper is structured as follows. In Section~\ref{Sec:Statement}, we describe a model thin graph-like junction, make assumptions for a given convection vector field and for boundary interactions, and formulate a problem. Section~\ref{Sec_Linear} is devoted to the construction of  a formal asymptotic expansion of the solution to the problem~\eqref{probl}. The asymptotic expansion consists of three parts: the regular part of the asymptotics located inside each thin cylinder,  the boundary-layer part located near the bases of some thin cylinders, and the node-layer part located in the vicinity of the node. Here we prove the solvability of all interrelated recurrent procedures that determine
 coefficients of these parts of the asymptotics. A complete asymptotic expansion $\mathfrak{U}^{(\varepsilon)}$ in the whole thin graph-like junction is constructed in Section~\ref{justification}, where we also calculate residuals that  its partial sum $\mathfrak{U}_{M}^{(\varepsilon)}$ leaves in the problem \eqref{probl}.  Our main result is Theorem \ref{Th_1} that justifies the constructed asymptotic expansion and provides both the asymptotic  uniform pointwise  and energy estimates  for the difference between the solution $u_\varepsilon$ and the partial sum for any $M \in \Bbb N$ and $M >  \frac{3}{2}(1 - \lfloor \alpha\rfloor).$  The article ends with a section of conclusions and remarks.

\section{Problem statement}\label{Sec:Statement}
For a small positive parameter $\varepsilon,$  a model thin graph-like junction  $\Omega_\varepsilon$  consists of three thin  cylinders
$$
\Omega_\varepsilon^{(i)} =
  \Bigg\{
  x=(x_1, x_2, x_3)\in\Bbb{R}^3 \colon \
  \varepsilon \,  \ell_0  <x_i<\ell_i, \quad
  \sum\limits_{j=1}^3 (1-\delta_{i,j})x_j^2<\varepsilon^2 h_i^2
  \Bigg\}, \quad i=1,2,3,
$$
that are connected with a domain $\Omega_\varepsilon^{(0)}$ (referred to  as the "node"). Here  $\ell_0\in(0, \frac13),$ $ \ell_i\geq1,$ $h_i>0$ are given numbers,  \  $\delta_{i,j}$ denotes  the Kronecker delta, i.e.,
$\delta_{ii} = 1$ and $\delta_{ij} = 0$ if $i \neq j.$

We denote the lateral surface of the thin cylinder $\Omega_\varepsilon^{(i)}$ by
$$
{\Gamma_\varepsilon^{(i)}} := \partial\Omega_\varepsilon^{(i)} \cap \{ x\in\Bbb{R}^3 \ : \ \varepsilon \ell_0<x_i<\ell_i \}
$$
and by
$$
\Upsilon_\varepsilon^{(i)} (\mu) := \Omega_\varepsilon^{(i)} \cap
\big\{  x\in\Bbb{R}^3 \, : \ x_i= \mu\big\}
$$
its cross-section at the point $\mu \in [ \varepsilon \ell_0, \ell_i].$
\begin{figure}[htbp]
\begin{center}
\includegraphics[width=5cm]{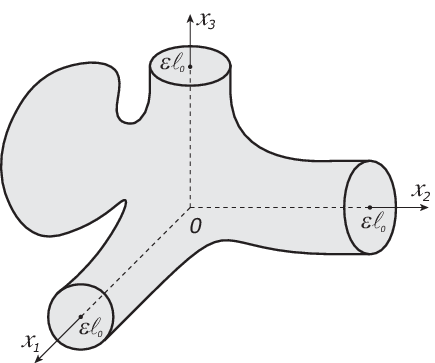}
\end{center}
\caption{The node $\Omega_\varepsilon^{(0)}$}\label{f2}
\end{figure}

The node $\Omega_\varepsilon^{(0)}$ (see Fig.~\ref{f2}) is formed by the homothetic transformation with coefficient $\varepsilon$ from a bounded domain $\Xi^{(0)}\subset \Bbb R^3$ containing the origin,  i.e.,
$
\Omega_\varepsilon^{(0)} = \varepsilon\, \Xi^{(0)}.
$
In addition, we assume that the boundary $\partial  \Xi^{(0)}$ of $\Xi^{(0)}$ contains the disks
$\Upsilon^{(i)}_1(\ell_0) := \overline{\Xi^{(0)}} \cap
\big\{ x\colon x_i= \ell_0\big\}, \   i\in\{1,2,3\},$ that are the bases of some right cylinders, respectively, and the lateral surfaces of these cylinders belong to $\partial  \Xi^{(0)}.$ So, the boundary of the node $\Omega_\varepsilon^{(0)}$ consists of
the disks $\Upsilon_\varepsilon^{(i)} (\varepsilon\ell_0), \   i\in\{1,2,3\},$
and the surface
$$
\Gamma_\varepsilon^{(0)} :=
\partial\Omega_\varepsilon^{(0)} \backslash
\left\{
 \overline{\Upsilon_\varepsilon^{(1)} (\varepsilon \ell_0)} \, \cup \,
 \overline{\Upsilon_\varepsilon^{(2)} (\varepsilon \ell_0)} \, \cup \,
 \overline{\Upsilon_\varepsilon^{(3)} (\varepsilon \ell_0)}
\right\}.
$$

Thus, $\Omega_\varepsilon$  is   the interior of the union
$
\bigcup_{i=0}^{3}\overline{\Omega_\varepsilon^{(i)}}
$
(see Fig.~\ref{f3}), and we assume that the surface $\partial\Omega_\varepsilon\setminus \bigcup_{i=0}^{3}\overline{\Upsilon_\varepsilon^{(i)} (\ell_i)}$ is smooth of the class $C^3.$

\begin{figure}[htbp]
\begin{center}
\includegraphics[width=8cm]{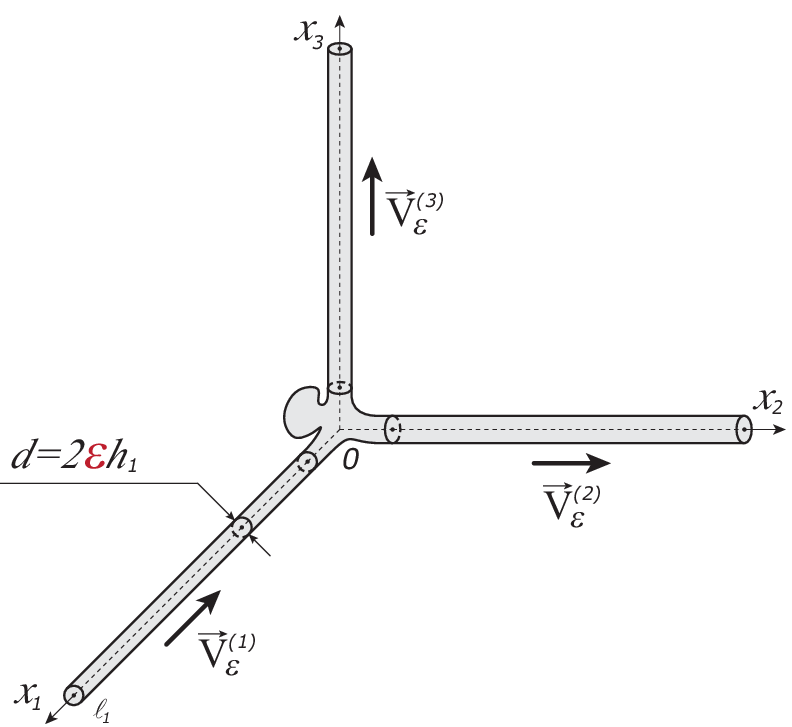}
\end{center}
\caption{The model thin graph-like  junction $\Omega_\varepsilon$}\label{f3}
\end{figure}

The given vector-valued function $\overrightarrow{V_\varepsilon}$ depends on the parts of $\Omega_\varepsilon$ and has the following structure:
\begin{gather}
\overrightarrow{V_\varepsilon}(x)=
\left(v^{(0)}_1(\tfrac{x}{\varepsilon}),  \  v^{(0)}_2(\tfrac{x}{\varepsilon}), \  v^{(0)}_3(\tfrac{x}{\varepsilon})
\right) =: \overrightarrow{V_\varepsilon}^{(0)}(x), \qquad x \in \ \Omega_\varepsilon^{(0)}, \notag
   \\
   \label{str_1}
\overrightarrow{V_\varepsilon}(x)=
\left(v^{(1)}_1(x_1), \ \varepsilon\,  v^{(1)}_2(x_1, \tfrac{\overline{x}_1}{\varepsilon}), \
 \varepsilon \, v^{(1)}_3(x_1,\tfrac{\overline{x}_1}{\varepsilon})
\right) =: \overrightarrow{V_\varepsilon}^{(1)}(x), \qquad x \in \ \Omega_\varepsilon^{(1)} \quad (v^{(1)}_1 < 0),
    \\
    \label{str_2}
\overrightarrow{V_\varepsilon}(x) =
\left( \varepsilon\,  v^{(2)}_1(x_2, \tfrac{\overline{x}_2}{\varepsilon}), \   v^{(2)}_2(x_2),  \
 \varepsilon \, v^{(2)}_3(x_2,\tfrac{\overline{x}_2}{\varepsilon})
\right) =: \overrightarrow{V_\varepsilon}^{(2)}(x), \qquad x \in \ \Omega_\varepsilon^{(2)} \quad (v^{(2)}_2 >  0),
     \\
     \label{st_3}
\overrightarrow{V_\varepsilon}(x) =
\left(\varepsilon\,  v^{(3)}_1(x_3, \tfrac{\overline{x}_3}{\varepsilon}), \
    \varepsilon \, v^{(3)}_2(x_3,\tfrac{\overline{x}_3}{\varepsilon}), \ v^{(3)}_3(x_3)
\right) =: \overrightarrow{V_\varepsilon}^{(3)}(x), \qquad x \in \ \Omega_\varepsilon^{(3)} \quad (v^{(3)}_3 >  0).
  \end{gather}
where
$$
\overline{x}_i =
\left\{\begin{array}{lr}
(x_2, x_3), & \text{if} \ \ i=1, \\
(x_1, x_3), & \text{if} \ \ i=2, \\
(x_1, x_2), & \text{if} \ \  i=3.
\end{array}\right.
$$
For a fixed value of the index $i\in \{1, 2, 3\}$, the function $v^{(i)}_i$ belongs to the space $C^3\big([0,\ell_i]\big)$ and is equal to a constant  $\mathrm{v}_i$ in a neighbourhood  of the origin; the other components of  $\overrightarrow{V_1}^{(i)}$ are smooth of class $C^3$ in $\overline{\Omega}_1^{(i)}$ and have  compact supports with respect to the  longitudinal  variable $x_i,$
 in particular, we will assume that they vanish in  $[0, \delta_i],$  $(\delta_i~>~0).$
Hence, the main direction of the vector field $\overrightarrow{V_\varepsilon}^{(i)}$ is oriented  along the axis of the cylinder $\Omega_\varepsilon^{(i)}$ (see Fig.~\ref{f3}).

To describe the dynamics of the velocity field $\overrightarrow{V_\varepsilon}$ in the node $\Omega_\varepsilon^{(0)},$ we make the following assumptions. The velocity field is conservative  in the node and
its potential $p$ is a solution to the boundary-value problem
\begin{equation}\label{potential0}
\left\{\begin{array}{rcll}
    \Delta_\xi p(\xi)  & = & 0, & \quad
    \xi \in \Xi^{(0)},
\\[2mm]
\dfrac{\partial p(\xi)}{\partial \xi_i}  &=&  \mathrm{v}_i, & \quad
   \xi \in \Upsilon^{(i)}_1(\ell_0), \quad i\in \{1, 2, 3\},
\\[4mm]
\dfrac{\partial p(\xi)}{\partial \boldsymbol{\nu}}  &=&  0, & \quad
   \xi \in \Gamma^{(0)} := \partial{\Xi^{(0)}} \setminus \Big( \bigcup_{i=1}^3 \Upsilon^{(i)}_1(\ell_0)\Big); \qquad \int_{\Xi^{(0)}} p(\xi)\, d\xi =0.
\end{array}\right.
\end{equation}
Here $\Delta_\xi$ is  the Laplace operator in the variables $\xi=(\xi_1,\xi_2,\xi_3),$  $\frac{\partial p}{\partial \boldsymbol{\nu}}$ is the derivative along the  outward unit normal $\boldsymbol{\nu}$ to $\partial \Xi^{(0)}.$  The Neumann problem \eqref{potential0} has a unique  solution if and only if the  conservation condition
\begin{equation}\label{cond_1}
  \sum_{i=1}^{3} h_i^2(0) \, \mathrm{v}_i =0
\end{equation}
holds. Thus,
\begin{equation}\label{field_0}
  \overrightarrow{V_\varepsilon}^{(0)}(x)  = \nabla_{\xi} p(\xi)\big|_{\xi = \frac{x}{\varepsilon}} = \varepsilon \nabla_x \big( p(\tfrac{x}{\varepsilon}) \big), \quad x\in \Omega^{(0)}_\varepsilon,
\end{equation}
and, clearly,  $\mathrm{div}_x \overrightarrow{V_\varepsilon}^{(0)} = 0$ in $\Omega^{(0)}_\varepsilon,$ i.e.,
 $\overrightarrow{V_\varepsilon}^{(0)}$  is incompressible in $\Omega^{(0)}_\varepsilon.$

In $\Omega_\varepsilon,$ we consider the following parabolic convection-diffusion  problem:
\begin{equation}\label{probl}
\left\{\begin{array}{rcll}
 \partial_t u_\varepsilon -  \varepsilon\, \Delta_x u_\varepsilon +
  \mathrm{div}_x \big( \overrightarrow{V_\varepsilon} \, u_\varepsilon\big)
  & = & 0, &
    \text{in} \ \Omega_\varepsilon \times (0, T),
\\[2mm]
-  \varepsilon \,  \partial_{\boldsymbol{\nu}_\varepsilon} u_\varepsilon +  u_\varepsilon \, \overrightarrow{V_\varepsilon}\boldsymbol{\cdot}\boldsymbol{\nu}_\varepsilon &  = & \varepsilon^{\alpha}  \varphi^{(i)}_\varepsilon &
   \text{on} \ \Gamma^{(i)}_\varepsilon \times (0, T), \quad i \in \{0, 1,2,3\},
   \\[2mm]
 u_\varepsilon \big|_{x \in \Upsilon_{\varepsilon}^{(i)} (\ell_i)}
 & = & q_i(t), &  t\in (0, T), \quad i \in \{1,2,3\},
\\[2mm]
u_\varepsilon\big|_{t=0}&=& 0, & \text{in} \ \Omega_{\varepsilon},
 \end{array}\right.
\end{equation}
where  ${\boldsymbol{\nu}}_\varepsilon$ is  the outward unit normal to $\partial \Omega_\varepsilon,$
$\partial_{\boldsymbol{\nu}_\varepsilon}$ denotes  the derivative along ${\boldsymbol{\nu}}_\varepsilon$,  the  intensity parameter $\alpha$  is less then $1.$ Given functions $\{\varphi^{(i)}_\varepsilon\}_{i=0}^3$ are determined as follows
$$
\varphi^{(0)}_\varepsilon(x,t) := \varphi^{(0)}\big(\tfrac{x}{\varepsilon},t\big), \quad (x,t) \in \overline{\Gamma_\varepsilon^{(0)}} \times [0,T],
$$
$$
 \varphi^{(i)}_\varepsilon(x,t) := \varphi^{(i)}\big(\tfrac{\overline{x}_i}{\varepsilon}, x_i, t\big), \quad (x,t) \in \overline{\Gamma_\varepsilon^{(i)}} \times [0,T], \quad i\in \{1,2,3\},
$$
where  the function $\varphi^{(0)}(\xi,t), \ (\xi,t) \in \overline{\Xi^{(0)}}\times [0,T],$  and the functions
$$
\varphi^{(i)}(\overline{\xi}_i,x_i,t), \ \ (\overline{\xi}_i,x_i,t) \in \big\{ |\overline{\xi}_i|\le h_i, \ x_i\in [0, \ell_i], \ t\in [0,T]\big\}, \quad  i\in\{1,2,3\},
$$
belong to the class $C^2$ in their domains of definition. In addition,   $\varphi^{(0)}$ vanishes uniformly with respect to $t\in [0, T]$ in neighborhoods of
$\Upsilon^{(i)}_1(\ell_0), \ i\in \{1,2,3\},$ the functions $\{\varphi^{(i)}\}_{i=1}^3$ vanish  uniformly with respect to $t$  and $\overline{\xi}_i$ in neighborhoods of the ends  of the corresponding closed interval  $[0, \ell_i].$

Given functions $\{q_i(t), \ t\in [0, T]\}_{i=1}^3 $ are smooth and nonnegative. To satisfy the zero- and first-order matching conditions for the problem \eqref{probl},  we assume that
\begin{equation}\label{match_conditions}
  q_i(0) =  \frac{d q_i}{dt}(0)=0, \quad i\in \{1,2,3\}, \quad \varphi^{(i)}\big|_{t=0} =0, \quad i\in\{0, 1,2,3\}.
\end{equation}

Thus, due to the classical theory of parabolic initial-boundary problem there exists  a unique classical solution to the problem \eqref{probl}  (see e.g. \cite[Chapt. IV, \S 5]{Lad_Sol_Ura_1968}). Obviously, it is also a weak solution in the Sobolev space $L^2\big(0,T; H^1(\Omega_\varepsilon)\big).$

Our goal is to construct the complete asymptotic expansion for  the solution to the problem \eqref{probl} as
$\varepsilon \to 0,$ i.e., when the thin junction $\Omega_\varepsilon$ is shrunk into the graph
$$
 \mathcal{I} := I_1 \cup I_2 \cup I_3,
$$
 where $I_i := \{x\colon  x_i \in  [0, \ell_i], \ \overline{x}_i  = (0, 0)\},$ $i\in \{1,2,3\}.$

\section{Formal asymptotics}\label{Sec_Linear}

To approximate the solution $u^\varepsilon$, the following Puiseux series approaches are suggested:
\begin{itemize}
  \item
   \begin{equation}\label{regul+alfa}
 \sum\limits_{k=0}^{+\infty} \sum\limits_{p=0}^{1}  \varepsilon^{p\alpha +k -1 } \, \left(w_{p\alpha +k -1}^{(i)} (x_i) + u_{p\alpha +k -1}^{(i)} \Big( x_i, \frac{\overline{x}_i}{\varepsilon}, t \Big) \right)
\end{equation}
for the regular part of the asymptotics in  each thin cylinder $\Omega^{(i)}_\varepsilon$ $(i\in \{1,2,3\});$
 \item
  \begin{equation}\label{junc+alfa}
\sum\limits_{k=0}^{+\infty} \sum\limits_{p=0}^{1}  \varepsilon^{p\alpha +k -1} \, N_{p\alpha +k -1}\left(\frac{x}{\varepsilon}, t\right)
\end{equation}
for the node-layer part of the asymptotics in a neighborhood of the node $\Omega^{(0)}_\varepsilon;$
\item
\begin{equation}\label{exp-EVl+alfa}
  \sum\limits_{k=0}^{+\infty} \sum\limits_{p=0}^{1}  \varepsilon^{p\alpha +k -1} \, \Pi^{(i)}_{p\alpha +k -1}\Big( \frac{\ell_i - x_i}{\varepsilon}, \dfrac{\overline{x}_i}{\varepsilon}, t \Big)
\end{equation}
for the boundary-layer part of the asymptotics in a neighborhood of the base $\Upsilon_\varepsilon^{(i)}(\ell_i)$ $(i \in \{2, 3\}).$
\end{itemize}

\begin{remark}\label{r4*1} The ansatzes \eqref{regul+alfa} -- \eqref{exp-EVl+alfa} are suitable for constructing the asymptotics of the solution to the problem~\eqref{probl} for any real value of the parameter $\alpha.$
Each of these ansatzes  is the formal sum of two ansatzes, namely
\begin{equation}\label{ansatz-structure}
  \sum\limits_{k=0}^{+\infty} \varepsilon^{k-1} \, \mathfrak{B}_{k-1} \ + \ \sum\limits_{k=0}^{+\infty} \varepsilon^{\alpha + k -1} \, \mathfrak{B}_{\alpha + k -1}.
\end{equation}
In the case of $\alpha \in \Bbb Z$, these ansatzes are transformed into ansatzes in integer powers of $\varepsilon,$ which have the following form:
$$
\sum\limits_{k=0}^{+\infty} \varepsilon^k \, \mathfrak{B}_k \quad \text{for} \ \alpha \in \Bbb N, \quad \text{and} \quad  \sum\limits_{k=\alpha -1}^{+\infty} \varepsilon^k \, \mathfrak{B}_k \quad \text{for} \  \alpha \in \Bbb Z, \ \alpha \le  0.
$$

In this paper  we consider the more complicated case $\alpha \in \Bbb R \setminus \Bbb Z$ and $\alpha < 1$,  always assuming that coefficients with negative integer indices and coefficients with indices less than $\alpha -1$ vanish in all series and equations.
\end{remark}

For such series, we give the following definition of an asymptotic expansion.

\begin{definition}\label{Puiseux} Let $\mathds{B}$  be a Banach space,  $f(\varepsilon)$ be an element in  $\mathds{B}$,
which depends on a small parameter $\varepsilon,$ and  $\alpha \in \Bbb R \setminus \Bbb Z,  \ \alpha < 1$.   We say that
a Puiseux series
\begin{equation}\label{Pu_ser}
    \sum\limits_{k=0}^{+\infty} \sum\limits_{p=0}^{1}  \varepsilon^{p\alpha +k -1} \, \mathfrak{b}_{p\alpha +k -1}
\end{equation}
whose coefficients belong to $\mathds{B}$  is the asymptotic expansion of $f(\varepsilon)$ in the Banach space $\mathds{B}$  if for any
$\mathfrak{N} > 0$ there exists a positive integer  $M_0$ such that for all  $M \in \Bbb N$ and $M \ge M_0$
\begin{equation}\label{def_as_exp}
\left\| f(\varepsilon) -  \mathcal{L}_M(\varepsilon)\right\|_{\mathds{B}} = \mathcal{O}(\varepsilon^\mathfrak{N}) \quad \text{as} \ \
\varepsilon \to 0,
\end{equation}
where  $\mathcal{L}_M(\varepsilon)$ is the partial sum  of \eqref{Pu_ser}, which  is defined as follows
\begin{equation}\label{part-sum}
  \mathcal{L}_M(\varepsilon) = \sum\limits_{k=0}^{- \lfloor \alpha\rfloor}   \varepsilon^{\alpha +k -1} \, \mathfrak{b}_{\alpha +k -1}
 +  \sum\limits_{k= 1}^{M +\lfloor \alpha\rfloor}   \varepsilon^{k -1} \, \mathfrak{b}_{k -1}  + \sum\limits_{k=- \lfloor \alpha\rfloor + 1}^{M}   \varepsilon^{\alpha +k -1} \, \mathfrak{b}_{\alpha +k -1},
\end{equation}
and $\lfloor \cdot \rfloor$ is the floor function.
The first sum $\sum_{k=0}^{- \lfloor \alpha\rfloor}   \varepsilon^{\alpha +k -1} \, \mathfrak{b}_{\alpha +k -1}$
 is called the principal part of the  Puiseux series \eqref{Pu_ser}.
\end{definition}
\begin{remark}
  It can be seen from Definition~\ref{Puiseux} that the number of  terms in the principal part of the  Puiseux series \eqref{Pu_ser} depends on the parameter $\alpha$ and each term in the principal part is unbounded as $\varepsilon \to 0.$  For example, if $\alpha \in (0, 1),$ then the principal part contains only one term $ \varepsilon^{\alpha  -1} \, \mathfrak{b}_{\alpha  -1}.$
\end{remark}

Formally substituting \eqref{regul+alfa}   into the  differential equation of the problem~\eqref{probl} and into  the boundary condition on the lateral surface of the thin cylinder $\Omega_\varepsilon^{(i)}$ (the index $i$ is  fixed), collecting coefficients at the same power of $\varepsilon$,  we get the following relations for each $k \in \Bbb N_0\cup \{-1\}$ and $ p\in \{0, 1\}$:
\begin{align}\label{eq_2}
 \Delta_{\bar{\xi}_i}u^{(i)}_{p\alpha +k}\big(x_i, \bar{\xi}_i, t\big)
  = &  \ \Big(  v^{(i)}_i(x_i) \, \big[ w^{(i)}_{p\alpha +k-1}(x_i, t) + u^{(i)}_{p\alpha +k-1}(x_i,\bar{\xi}_i, t) \big] \Big)^\prime
  +  \partial_t {w}_{p\alpha +k-1}^{(i)}(x_i, t)   \notag
    \\
    + & \  \partial_t {u}_{p\alpha +k-1}^{(i)}(x_i,\bar{\xi}_i, t) + \mathrm{div}_{\bar{\xi}_i} \Big( \overline{V}^{(i)}(x_i, \bar{\xi}_i) \,
            \big[ w^{(i)}_{p\alpha +k-1}(x_i, t) + u^{(i)}_{p\alpha +k-1}(x_i,\bar{\xi}_i, t) \big]\Big)  \notag
            \\
             - & \ \Big( w^{(i)}_{p\alpha +k-2}(x_i,t)  + u^{(i)}_{p\alpha +k-2}(x_i,\bar{\xi}_i, t) \Big)^{\prime\prime},
\quad  \bar{\xi}_i \in \Upsilon_i,
\end{align}
where  $\bar{\xi}_i=\frac{\overline{x}_i}{\varepsilon},$  $\Upsilon_i := \big\{ \overline{\xi}_i\in\Bbb{R}^2 \colon |\overline{\xi}_i|< h_i \big\},$  ``$\prime$'' denotes the derivative with respect to the longitudinal variable $x_i,$
$$
\overline{V}^{(1)}:= \Big(v^{(1)}_2(x_1, \bar{\xi}_1), \, v^{(1)}_3(x_1,\bar{\xi}_1)\Big), \ \
\overline{V}^{(2)}:= \left(v^{(2)}_1(x_2, \bar{\xi}_2), \,   v^{(2)}_3(x_2,\bar{\xi}_2)\right), \ \ \overline{V}^{(3)}:=
 \left(v^{(3)}_1(x_3, \bar{\xi}_3), \, v^{(3)}_2(x_3,\bar{\xi}_3)\right);
 $$
 and
\begin{equation}\label{bc_2}
   \partial_{\bar{\nu}_{\bar{\xi}_i}} u^{(i)}_{p\alpha +k} = \big( w^{(i)}_{p\alpha +k -1} + u^{(i)}_{p\alpha +k -1} \big) \, \overline{V}^{(i)} \boldsymbol{\boldsymbol{\cdot}} \bar{\nu}_{\bar{\xi}_i} \,  - \,  \delta_{p\alpha +k -1, \, \alpha -1} \, \varphi^{(i)}\big(\bar{\xi}_i, x_i, t\big), \quad \bar{\xi}_i \in \partial \Upsilon_i,
\end{equation}
where $\partial_{\bar{\nu}_{\bar{\xi}_i}} $ is the derivative along the  outward unit normal $\bar{\nu}_{\bar{\xi}_i}$ to the boundary of the disk $\Upsilon_i,$  $\delta_{p\alpha +k -1, \, \alpha -1}$ is   the Kronecker delta, $\overline{V}^{(i)} \boldsymbol{\cdot} \bar{\nu}_{\bar{\xi}_i}$ is the scalar product of the vectors $\overline{V}^{(i)}$ and $\bar{\nu}_{\bar{\xi}_i}.$
\begin{remark}\label{r4*2}
In \eqref{bc_2} and below, we will omit the dependence of some coefficients on variables to simplify the writing of equations, if this does not cause confusion.
\end{remark}

The equations  \eqref{eq_2} and \eqref{bc_2} form the Neumann problems in $\Upsilon_i$ with respect to the variables $\bar{\xi}_i$ to find $u^{(i)}_{p\alpha +k}.$ In the right-hand sides of \eqref{eq_2} and \eqref{bc_2}  there are coefficients of the
ansatz \eqref{regul+alfa} with lower priorities. The variables $x_i$ and $t$ in these problems are considered as parameters from the set $I_\varepsilon^{(i)} \times (0, T)$, where $ I_\varepsilon^{(i)} := \{x_i\colon x_i \in (\varepsilon \ell_0, \ell_i), \ \bar{x}_i = (0, 0) \}.$
To ensure  uniqueness, we provide each  problem with the  condition
\begin{equation}\label{uniq_1}
  \langle u_{p\alpha +k}^{(i)}(x_1, \, \cdot \, ,  t ) \rangle_{\Upsilon_i} :=  \int_{\Upsilon_i} u_{p\alpha +k}^{(i)}\big(x_i, \bar{\xi}_i, t\big) \, d\bar{\xi}_i = 0.
\end{equation}

Considering the last part of Remark~\ref{r4*1}, the Neumann problem for $u^{(i)}_{\alpha -1}$ $(k=-1, \ p=1)$ is as follows
\begin{equation}\label{u_0}
  \Delta_{\bar{\xi}_i}u^{(i)}_{\alpha -1} = 0 \ \ \text{in } \ \ \Upsilon_i, \quad \partial_{\bar{\nu}_{\bar{\xi}_i}} u^{(i)}_{\alpha -1} = 0 \ \ \text{on} \ \ \partial \Upsilon_i, \quad \langle u_{\alpha -1}^{(i)}\rangle_{\Upsilon_i} = 0,
\end{equation}
whence $u^{(i)}_{\alpha -1} \equiv 0.$ Similarly, we obtain $u^{(i)}_{0} \equiv 0$ $(k= p=0).$

Thus, the ansatz  \eqref{regul+alfa} begins with the summand $\varepsilon^{\alpha -1} w^{(i)}_{\alpha -1}.$

For $k=0, \ p=1$ in \eqref{eq_2} and \eqref{bc_2},  we have the problem
\begin{equation}\label{p_2}
\left\{\begin{array}{rcl}
    \Delta_{\bar{\xi}_i}u^{(i)}_{\alpha} & = & \Big(v^{(i)}_i \,  w^{(i)}_{\alpha-1}\Big)^\prime
  +  \partial_t {w}_{\alpha-1}^{(i)} + {w}_{\alpha-1}^{(i)} \, \mathrm{div}_{\bar{\xi}_i}\overline{V}^{(i)} \quad  \text{in } \ \ \Upsilon_i,
\\[2mm]
\partial_{\bar{\nu}_{\bar{\xi}_i}} u^{(i)}_{\alpha}  &=&  w^{(i)}_{\alpha-1} \ \overline{V}^{(i)} \boldsymbol{\cdot} \bar{\nu}_{\bar{\xi}_i} - \varphi^{(i)}\big(\bar{\xi}_i, x_i, t\big)  \ \ \text{on} \ \ \partial \Upsilon_i, \qquad \langle u_{\alpha}^{(i)}\rangle_{\Upsilon_i} = 0;
\end{array}\right.
 \end{equation}
and for $k=1, \ p= 0$ the problem
\begin{equation}\label{p_1}
\left\{\begin{array}{rcl}
    \Delta_{\bar{\xi}_i}u^{(i)}_{1} & = & \Big(v^{(i)}_i \,  w^{(i)}_{0}\Big)^\prime
  +  \partial_t {w}_{0}^{(i)} + {w}_{0}^{(i)} \, \mathrm{div}_{\bar{\xi}_i}\overline{V}^{(i)} \quad  \text{in } \ \ \Upsilon_i,
\\[2mm]
\partial_{\bar{\nu}_{\bar{\xi}_i}} u^{(i)}_{1}  &=&  w^{(i)}_0 \ \overline{V}^{(i)} \boldsymbol{\boldsymbol{\cdot}} \bar{\nu}_{\bar{\xi}_i} \ \ \text{on} \ \ \partial \Upsilon_i, \qquad \langle u_{1}^{(i)}\rangle_{\Upsilon_i} = 0.
\end{array}\right.
 \end{equation}

Writing down the solvability condition for each of these problems, we deduce the differential equations for the coefficients  $w^{(i)}_{\alpha-1}$ and $w^{(i)}_0$:
\begin{equation}\label{lim_1}
  \partial_t{w}^{(i)}_{\alpha-1}(x_i,t) \, + \, \big( v_i^{(i)}(x_i)\, w^{(i)}_{\alpha-1}(x_i,t) \big)^\prime
  =  - \widehat{\varphi}^{(i)}(x_i, t), \quad (x_i, t) \in I_\varepsilon^{(i)} \times (0, T),
  \end{equation}
and
\begin{equation}\label{lim_0}
  \partial_t{w}^{(i)}_0(x_i,t) \, + \, \big( v_i^{(i)}(x_i)\, w^{(i)}_0(x_i,t) \big)^\prime = 0, \quad (x_i, t) \in I_\varepsilon^{(i)} \times (0, T),
\end{equation}
where
\begin{equation}\label{hat_phi}
  \widehat{\varphi}^{(i)}(x_i, t) := \frac{1}{\pi h^2_i} \int_{\partial \Upsilon_i} \varphi^{(i)}(x_i, \bar{\xi}_i, t\big)\, d\sigma_{\bar{\xi}_i}.
\end{equation}
Let   ${w}^{(i)}_{\alpha-1}$ and ${w}^{(i)}_0$  be solutions to the equations  \eqref{lim_1} and \eqref{lim_0}, respectively. We will show below how to choose suitable unique solutions. Then there exist unique solutions to the problems \eqref{p_1} and \eqref{p_2}, respectively, and the differential equations in these problems become
$$
 \Delta_{\bar{\xi}_i}u^{(i)}_{\alpha}  =  - \widehat{\varphi}^{(i)} + {w}_{\alpha-1}^{(i)} \, \mathrm{div}_{\bar{\xi}_i}\overline{V}^{(i)} \quad \text{and} \quad
 \Delta_{\bar{\xi}_i}u^{(i)}_{1}  =   {w}_{0}^{(i)} \, \mathrm{div}_{\bar{\xi}_i}\overline{V}^{(i)}.
$$

Writing down the solvability condition for the problem \eqref{eq_2}--\eqref{uniq_1}, we get the differential equation
\begin{equation}\label{lim_2}
  \partial_t {w}_{p\alpha +k-1}^{(i)}(x_i, t)   + \Big(  v^{(i)}_i(x_i) \,  w^{(i)}_{p\alpha +k-1}(x_i, t) \Big)^\prime
  = \Big( w^{(i)}_{p\alpha +k-2}(x_i,t)\Big)^{\prime\prime},  \quad (x_i, t) \in I_\varepsilon^{(i)} \times (0, T).
\end{equation}
 Let ${w}_{p\alpha +k-1}^{(i)}$ be a solution of \eqref{lim_2}. Then there exists a unique solution to the problem \eqref{eq_2}--\eqref{uniq_1} and the differential equation in this problem becomes
 \begin{align}\label{eq_2+}
 \Delta_{\bar{\xi}_i}u^{(i)}_{p\alpha +k}\big(x_i, \bar{\xi}_i, t\big)
  = &  \ \Big(  v^{(i)}_i(x_i) \,  u^{(i)}_{p\alpha +k-1}(x_i,\bar{\xi}_i, t)  \Big)^\prime
  +    \partial_t {u}_{p\alpha +k-1}^{(i)}(x_i,\bar{\xi}_i, t) -  \Big( u^{(i)}_{p\alpha +k-2}(x_i,\bar{\xi}_i, t) \Big)^{\prime\prime}  \notag
    \\
    + & \   \mathrm{div}_{\bar{\xi}_i} \Big( \overline{V}^{(i)}(x_i, \bar{\xi}_i) \,
            \big[ w^{(i)}_{p\alpha +k-1}(x_i, t) + u^{(i)}_{p\alpha +k-1}(x_i,\bar{\xi}_i, t) \big]\Big), \quad  \bar{\xi}_i \in \Upsilon_i.
\end{align}

\begin{remark}\label{r_2_1}
Since the functions $\{\varphi^{(i)}\}_{i=1}^3$ and $\{\overline{V}^{(i)}\}_{i=1}^3$  have compact supports with respect to the corresponding longitudinal variable, the coefficients $\{u_{p\alpha +k -1}^{(i)}\}$ vanish in the corresponding neighborhoods
of the ends of the segment $[0, \ell_i].$
\end{remark}

To find transmission conditions for  the functions $\{w_{p\alpha +k -1}^{(i)}\}, \ i\in \{1,2,3\}$ at the point $0,$ we should run the node-layer  part~ (\ref{junc+alfa}) of the asymptotics in a neighborhood of the node $\Omega^{(0)}_\varepsilon$. For this purpose we pass to the scaled variables $\xi=\frac{x}{\varepsilon}.$ Then, letting  $\varepsilon$ to $0,$ we see  that the domain $\Omega_\varepsilon$ is transformed into the unbounded domain $\Xi$ that  is the union of the domain~$\Xi^{(0)}$ and three semi-infinite cylinders
$$
\Xi^{(i)}  =  \{ \xi=(\xi_1,\xi_2,\xi_3)\in\Bbb R^3 \colon  \   \ell_0<\xi_i<+\infty,
    \quad |\overline{\xi}_i|<h_i\}, \quad i\in \{1,2,3\},
$$
i.e., $\Xi$ is the interior of the set $\bigcup_{i=0}^3\overline{\Xi^{(i)}}.$

Next, repeating the same procedure with the ansatz  \eqref{junc+alfa}, considering the incompressibility of $\overrightarrow{V_\varepsilon}^{(0)}$ in $\Omega^{(0)}_\varepsilon,$ and   matching  ansatzes \eqref{regul+alfa} and  \eqref{junc+alfa},
we get the problem
\begin{equation}\label{N_p_alpha +k_prob}
\left\{\begin{array}{rclll}
-   \Delta_\xi {N}_{p\alpha +k}(\xi,t) +
  \overrightarrow{V}(\xi) \boldsymbol{\cdot} \nabla_\xi {N}_{p\alpha +k}(\xi,t) & = &  - \partial_t {{N}}_{p\alpha +k-1}(\xi,t), &  \xi \in\Xi^{(0)},
\\[2mm]
- \partial_{\boldsymbol{\nu}_\xi}  N_{p\alpha +k}(\xi,t) &=& \delta_{p\alpha +k, \alpha }\varphi^{(0)}(\xi,t) , &   \xi \in \Gamma_0,
\\[2mm]
 -   \Delta_{\xi} {N}_{p\alpha +k}(\xi,t)  +
  \mathrm{v}_i \, \partial_{\xi_i}{N}_{p\alpha +k}(\xi,t) & = &  - \partial_t{{N}}_{p\alpha +k-1}(\xi,t), &
     \xi \in\Xi^{(i)},
\\[2mm]
\partial_{\bar{\nu}_{\bar{\xi}_i}} N_{p\alpha +k}(\xi,t)  &=&  0, &
   \xi \in \Gamma_i,
\\[2mm]
N_{p\alpha +k}(\xi,t)     \  \sim \   w^{(i)}_{p\alpha +k}(0,t) + \Psi^{(i)}_{p\alpha +k}(\xi_i,t)      &\text{as}  &\xi_i \to +\infty, &
    \xi  \in \Xi^{(i)}, \ \ i\in \{1,2,3\},
 \end{array}\right.
\end{equation}
 for each $k \in \Bbb N_0\cup \{-1\}$ and $ p\in \{0, 1\},$ where
\begin{equation}\label{Psi}
\Psi_{ p\alpha+k}^{(i)}(\xi_i,t) =   \sum\limits_{j=1}^{k+1} \frac{\xi_i^j}{j!}\,
     \dfrac{\partial^j w_{p\alpha + k -j}^{(i)}}{ \partial x_i^j} (0,t)
 \quad  i\in\{1,2,3\}.
\end{equation}
Note that the variable $t$ appears as a parameter  in the  steady convection-diffusion problems \eqref{N_p_alpha +k_prob}.

A solution with such a polynomial asymptotics at different exits to infinity is sought in the form
\begin{equation}\label{new-solution_p_alpha +k}
N_{p\alpha +k}(\xi,t)  = \sum\limits_{i=1}^3\big(w^{(i)}_{p\alpha +k}(0,t) + \Psi^{(i)}_{p\alpha +k}(\xi_i,t) \big) \,\chi_{\ell_0}(\xi_i) + \widetilde{N}_{p\alpha +k}(\xi,t),
\end{equation}
where $ \chi_{\ell_0} \in C^{\infty}(\Bbb{R})$ is a smooth cut-off function such that
$\ 0\leq \chi_{\ell_0} \leq1,$  $\chi_{\ell_0}(s) =0$ if $s \leq  2\ell_0$  and
$\chi_{\ell_0}(s) =1$ if $s \geq  3\ell_0.$ Then $\widetilde{N}_{p\alpha +k}$ must be a solution to the problem
\begin{equation}\label{N_p_alpha +k_prob+}
\left\{\begin{array}{rclll}
-   \Delta_\xi {\widetilde{N}}_{p\alpha +k} +
  \overrightarrow{V}(\xi) \boldsymbol{\cdot} \nabla_\xi {\widetilde{N}}_{p\alpha +k} & = &  - \partial_t {{N}}_{p\alpha +k-1}(\xi,t), &  \xi \in\Xi^{(0)},
\\[2mm]
- \partial_{\boldsymbol{\nu}_\xi}  \widetilde{N}_{p\alpha +k}(\xi,t) &=& \delta_{p\alpha +k, \alpha }\varphi^{(0)}(\xi,t) , &   \xi \in \Gamma_0,
\\[2mm]
 -   \Delta_{\xi} \widetilde{N}_{p\alpha +k}  +
  \mathrm{v}_i \, \partial_{\xi_i}\widetilde{N}_{p\alpha +k} & = & F_{p\alpha +k}(\xi_i,t) - \partial_t{{N}}_{p\alpha +k-1}(\xi,t), &
     \xi \in\Xi^{(i)},
\\[2mm]
\partial_{\bar{\nu}_{\bar{\xi}_i}} \widetilde{N}_{p\alpha +k}(\xi,t)  &=&  0, &
   \xi \in \Gamma_i,
\\[2mm]
\widetilde{N}_{p\alpha +k}(\xi,t)     \  \rightarrow  \   0  &\text{as}  &\xi_i \to +\infty, &
    \xi  \in \Xi^{(i)}, \ \ i\in \{1,2,3\},
 \end{array}\right.
\end{equation}
where
\begin{align*}
  F_{p\alpha +k}(\xi_i,t) := & \ w^{(i)}_{p\alpha +k}(0,t) \,\chi''_{\ell_0}(\xi_i) - \mathrm{v}_i \, w^{(i)}_{p\alpha +k}(0,t) \,\chi'_{\ell_0}(\xi_i)
     \\
   & +  \ \partial^2_{\xi_i \xi_i}\Big(\Psi^{(i)}_{p\alpha +k}(\xi_i,t) \big) \,\chi_{\ell_0}(\xi_i)\Big) -  \mathrm{v}_i \, \partial_{\xi_i}\Big(\Psi^{(i)}_{p\alpha +k}(\xi_i,t) \big) \,\chi_{\ell_0}(\xi_i)\Big).
\end{align*}

Proposition~A.1 from \cite[Appendix A]{Mel-Roh_preprint-2023} asserts  that the necessary and sufficient condition for the  unique solvability of the problem \eqref{N_p_alpha +k_prob+} in the Sobolev space of functions exponentially decreasing to zero is the equality
\begin{equation}\label{solv_cond}
\sum_{i=1}^{3}\int_{\Xi^{(i)}} \big(F_{p\alpha +k} - \partial_t{{N}}_{p\alpha +k-1}\big)\, d\xi = \int_{\Xi^{(0)}} \partial_t {{N}}_{p\alpha +k-1} \, d\xi + \delta_{p\alpha +k, \alpha }\int_{\Gamma_0} \varphi^{(0)}\, d \sigma_\xi.
\end{equation}

Since
$$
\partial^2_{\xi_i \xi_i}\Psi^{(i)}_{p\alpha +k}(\xi_i,t)  -  \mathrm{v}_i \, \partial_{\xi_i}\Psi^{(i)}_{p\alpha +k}(\xi_i,t) =
\partial_t w^{(i)}_{p\alpha +k-1}(0,t) + \partial_t\Psi^{(i)}_{p\alpha +k-1}(\xi_i,t),
$$
the difference
\begin{gather}\label{F_alpha}
F_{p\alpha +k}(\xi_i,t) - \partial_t{{N}}_{p\alpha +k-1}(\xi,t) = w^{(i)}_{p\alpha +k}(0,t) \,\chi''_{\ell_0}(\xi_i) - \mathrm{v}_i \, w^{(i)}_{p\alpha +k}(0,t) \,\chi'_{\ell_0}(\xi_i)
\notag
  \\
  +\Big(\partial_{\xi_i}\Psi^{(i)}_{p\alpha +k}(\xi_i,t)   -
\mathrm{v}_i \, \Psi^{(i)}_{p\alpha +k}(\xi_i,t) \Big)\, \chi'_{\ell_0}(\xi_i) + \Big(\Psi^{(i)}_{p\alpha +k}(\xi_i,t) \, \chi'_{\ell_0}(\xi_i)\Big)'
- \partial_t{\widetilde{N}}_{p\alpha +k-1}(\xi,t).
\label{F_alpha}
\end{gather}

Using \eqref{F_alpha}, the equality \eqref{solv_cond}  can be rewritten as the gluing condition for the functions $\{w^{(i)}_{p\alpha +k}\}_{i=1}^3$  at the origin
\begin{equation}\label{trans*k-p*alpfa}
    \sum\limits_{i=1}^3  \mathrm{v}_i  \, h_i^2\, w_{p\alpha +k}^{(i)} (0,t)
 =  \boldsymbol{d_{p\alpha +k}}(t),
\end{equation}
where
\begin{align}\label{const_d*k-p*alpfa}
\boldsymbol{d_{p\alpha +k}}(t)
 = & \  - \delta_{p\alpha +k, \alpha }\, \frac{1}{\pi}\int_{\Gamma_0} \varphi^{(0)}\, d \sigma_\xi  -\frac{1}{\pi} \int_{\Xi^{(0)}} \partial_t {{N}}_{p\alpha +k-1} \, d\xi  - \frac{1}{\pi}  \sum_{i=1}^{3}\int_{\Xi^{(i)}} \partial_t{\widetilde{N}}_{p\alpha +k-1}\, d\xi \notag
 \\
 &\ +\sum_{i=1}^{3} h_i^2 \int_{2 \ell_0}^{3 \ell_0} \big(\partial_{\xi_i}\Psi^{(i)}_{p\alpha +k}(\xi_i,t)   -
\mathrm{v}_i \, \Psi^{(i)}_{p\alpha +k}(\xi_i,t) \big)\, \chi'_{\ell_0}(\xi_i) d\xi.
\end{align}

For $p=1, \ k = -1$ and $p=0, \ k=0$ the value $\boldsymbol{d_{p\alpha +k}} \equiv 0.$ Thus, to determine   $\{w^{(i)}_{\alpha-1}\}_{i=1}^3$ and $\{w^{(i)}_0\}_{i=1}^3$ we get the following problems on the graph $\mathcal{I}$:
\begin{equation}\label{limit_al-1}
\left\{\begin{array}{l}
   \partial_t{w}^{(i)}_{\alpha-1}(x_i,t) \, + \, \big( v_i^{(i)}(x_i)\, w^{(i)}_{\alpha-1}(x_i,t) \big)^\prime
  =  - \widehat{\varphi}^{(i)}(x_i, t), \quad (x_i, t) \in (0, \ell_i)\times (0, T), \ \ i\in \{1, 2, 3\},
\\[2mm]
 \sum_{i=1}^3  \mathrm{v}_i  \, h_i^2\, w_{\alpha-1}^{(i)} (0,t) = 0 \ \ \text{for any} \ \ t \in [0, T],
 \\[2mm]
 w^{(1)}_{\alpha-1}(\ell_1,t) = 0\ \ \text{for any} \ \ t \in [0, T], \qquad {w}^{(i)}_{\alpha-1}(x_i,0) =0 \ \ \text{for any} \ \ x_i\in [0, \ell_i], \ \ i\in \{1, 2, 3\} ;
 \end{array}\right.
 \end{equation}
and
\begin{equation}\label{limit_0}
\left\{\begin{array}{l}
    \partial_t{w}^{(i)}_0(x_i,t) \, + \, \big( v_i^{(i)}(x_i)\, w^{(i)}_0(x_i,t) \big)^\prime = 0, \quad (x_i, t) \in (0, \ell_i)\times (0, T), \quad i\in \{1, 2, 3\},
\\[2mm]
 \sum_{i=1}^3  \mathrm{v}_i  \, h_i^2\, w_{0}^{(i)} (0,t) = 0 \ \ \text{for any} \ \ t \in [0, T],
 \\[2mm]
  w^{(1)}_0(\ell_1,t) =q_1(t) \ \ \text{for any} \ \ t \in [0, T], \qquad {w}^{(i)}_0(x_i,0) =0 \ \ \text{for any} \ \ x_i\in [0, \ell_i], \ \ i\in \{1, 2, 3\}.
 \end{array}\right.
 \end{equation}
In these problems, there is only one boundary condition at the end $x_1 =\ell_1$ because there is only one input cylinder with respect to the vector field $\overrightarrow{V_\varepsilon}.$ This is in full agreement with the approach proposed in \cite[\S 3.2.1]{Mel-Roh_preprint-2022}. According to this approach, we first find a solution to the corresponding hyperbolic mixed problem with a given boundary condition and initial condition. For example, for \eqref{limit_al-1} this is the following problem:
\begin{equation}\label{mixedt_al-1}
\left\{\begin{array}{l}
   \partial_t{w}^{(i)}_{\alpha-1}(x_1,t) \, + \, \big( v_1^{(1)}(x_1)\, w^{(1)}_{\alpha-1}(x_1,t) \big)^\prime
  =  - \widehat{\varphi}^{(1)}(x_1, t), \quad (x_1, t) \in (0, \ell_1)\times (0, T),
 \\[2mm]
 w^{(1)}_{\alpha-1}(\ell_1,t) = 0\ \ \text{for any} \ \ t \in [0, T], \qquad {w}^{(1)}_{\alpha-1}(x_1,0) =0 \ \ \text{for any} \ \ x_1\in [0, \ell_1],
 \end{array}\right.
 \end{equation}
The solvability criteria base on the characteristics method. Since $\widehat{\varphi}^{(1)}(\ell_1,0) =0$, the problem  \eqref{mixedt_al-1}
has a unique classical solution; in addition the explicit  representation is possible  for it  (see \cite[\S 3.2.1]{Mel-Roh_preprint-2022}). From this representation it follows that $\partial_t w_{\alpha-1}^{(1)}(0,0)=0.$

Then  ${w}^{(2)}_{\alpha-1}$ and ${w}^{(3)}_{\alpha-1}$ are defined as  classical solutions to the problems
\begin{equation}\label{mixedt_al-1}
\left\{\begin{array}{l}
     \partial_t{w}^{(i)}_{\alpha-1}(x_i,t) \, + \, \big( v_i^{(i)}(x_i)\, w^{(i)}_{\alpha-1}(x_i,t) \big)^\prime
  =  - \widehat{\varphi}^{(i)}(x_i, t), \quad (x_i, t) \in (0, \ell_i)\times (0, T),
 \\[2mm]
 w^{(i)}_{\alpha-1}(0,t) = w^{(1)}_{\alpha-1}(0,t) \ \ \text{for any} \ \ t \in [0, T], \qquad {w}^{(i)}_{\alpha-1}(x_i,0) =0 \ \ \text{for any} \ \ x_i\in [0, \ell_i],
 \end{array}\right.
 \end{equation}
 $i\in \{2, 3\},$ respectively. Since $ w_{\alpha-1}^{(1)}(0,0) = \partial_t w_{\alpha-1}^{(1)}(0,0)=0$ and $\widehat{\varphi}^{(2)}(0, 0) =\widehat{\varphi}^{(3)}(0, 0) =0,$ the solvability conditions are satisfied for these problem. Moreover, thanks to \eqref{cond_1}, the gluing condition at the origin is fulfilled in the problem \eqref{limit_al-1}.

 Thus, the  problem \eqref{limit_al-1} has a unique classical solution, and, interestingly, the continuity condition and the Kirchhoff condition are simultaneously satisfied at the graph vertex. In the same way, we justify the  existence and uniqueness of the classical solution to the problem \eqref{limit_0} (now  due to \eqref{match_conditions}), for which  the continuity  condition and the Kirchhoff condition are also satisfied.

 Having found these solutions, we can uniquely determine both solutions to the problems \eqref{p_2} and \eqref{p_1} and  solutions to the problems
\begin{equation}\label{N_p_alpha-1+0}
\left\{\begin{array}{rclll}
-   \Delta_\xi {\widetilde{N}}_{p\alpha +k} +
  \overrightarrow{V}(\xi) \boldsymbol{\cdot} \nabla_\xi {\widetilde{N}}_{p\alpha +k} & = & 0&  \text{in}\ \ \Xi^{(0)},
\\[2mm]
 -   \Delta_{\xi} \widetilde{N}_{p\alpha +k}  +
  \mathrm{v}_i \, \partial_{\xi_i}\widetilde{N}_{p\alpha +k} & = & w^{(i)}_{p\alpha +k}(0,t) \big(\chi''_{\ell_0}(\xi_i) - \mathrm{v}_i \chi'_{\ell_0}(\xi_i)\big) & \text{in}\ \ \Xi^{(i)},
\\[2mm]
- \partial_{\boldsymbol{\nu}_\xi}  \widetilde{N}_{p\alpha +k} &=& 0 &   \text{on}\  \ \partial\Xi,
\\[2mm]
\widetilde{N}_{p\alpha +k}(\xi,t)     \  \rightarrow  \   0  &\text{as}  &\xi_i \to +\infty, &
    \xi  \in \Xi^{(i)}, \ \ i\in \{1,2,3\},
 \end{array}\right.
\end{equation}
for $p=1, \ k= -1$ and for $p=k=0$, respectively. Since the right-hand sides  in the problem \eqref{N_p_alpha-1+0} are uniformly bounded with respect to $(\xi, t)\in \Xi\times [0, T]$ and have compact supports,  the corresponding solutions ${N}_{\alpha-1}$ and $N_{0}$ have  the following  asymptotics uniform with respect to $t\in [0, T]$:
\begin{equation}\label{rem_exp-decrease+0}
N_{p\alpha +k}(\xi,t) = w^{(i)}_{p\alpha +k}(0,t)  +  \mathcal{ O}(\exp(-\beta_0\xi_i))
\quad \mbox{as} \ \ \xi_i\to+\infty,  \ \  \xi  \in \Xi^{(i)},  \quad i=\{1,2,3\} \quad (\beta_0 >0)
\end{equation}
for $p=1, \ k= -1$ and $p=k=0$. It is easy to verify that $\partial_t{N}_{\alpha-1}$ and $\partial_t N_{0}$ have similar uniform asymptotics as well. In addition,
\begin{equation*}
  N_{p\alpha +k}\big|_{t=0} \equiv \widetilde{N}_{p\alpha +k}\big|_{t=0} \equiv 0, \quad \partial_t N_{p\alpha +k}\big|_{t=0} \equiv \partial_t \widetilde{N}_{p\alpha +k}\big|_{t=0} \equiv 0, \quad u_{p\alpha +k+1}^{(i)}\big|_{t=0} \equiv 0, \quad i\in \{1, 2, 3\}
\end{equation*}
for $p=1, \ k= -1$ and $p=k=0$.

\begin{remark}\label{additional}
It follows from the differential equations \eqref{lim_2} and \eqref{eq_2+}  that  $\{w_{p\alpha +k -1}^{(i)}\}$ and $\{u_{p\alpha +k}^{(i)}\}$
are defined in terms of the second derivatives of $\{w_{p\alpha +k -2}^{(i)}\}$ and $\{u_{p\alpha +k-2}^{(i)}\}.$
This means that $\{w^{(i)}_{\alpha-1}\}_{i=1}^3$ and $\{w^{(i)}_0\}_{i=1}^3$, as well as the other coefficients, must be infinitely differentiable.
Therefore, additional smoothness of the given functions and additional matching conditions are necessary, namely
\begin{itemize}
  \item $\{\varphi^{(i)}, \ q_i, \overline{V}^{(i)}, v^{(i)}_i\}_{i=1}^3$  belong to the class $C^\infty$ in their domains of definition, the function $\varphi^{(0)}$ is infinitely differentiable in $t\in [0, T],$
  \item for each $n\in \Bbb N$ and $i\in \{1, 2, 3\}$
  \begin{equation}\label{match_cond+}
   \frac{d^n q_1 }{dt^n}\Big|_{t=0} =  0\quad \text{and} \quad \frac{\partial^n\varphi^{(0)}}{\partial t^n}\Big|_{t=0}= 0.
\end{equation}
\end{itemize}
\end{remark}

Using the explicit  representation and additional assumptions, it can be verified that  $\partial^n_t w_{\alpha-1}^{(i)}(0, 0) = 0$  and $\partial^n_t w_{0}^{(i)}(0, 0) = 0$ for each $n\in \Bbb N.$ This means also that $\partial^n_t N_{p\alpha +k}\big|_{t=0} \equiv \partial^n_t \widetilde{N}_{p\alpha +k}\big|_{t=0} \equiv 0$ for $p=1, \ k= -1$ and $p=k=0$.

The influence of the interaction on the node boundary begins to manifest itself from the coefficient  $N_{\alpha }$ that is a solution to the problem \eqref{N_p_alpha +k_prob} for $p=1, \ k= 0.$ In this case, the solvability condition \eqref{trans*k-p*alpfa} for
$\widetilde{N}_{\alpha }$ has the non-zero right-hand side
\begin{align}\label{const_d_alpfa}
\boldsymbol{d_{\alpha}}(t)
 = & \ -  \frac{1}{\pi}\int_{\Gamma_0} \varphi^{(0)}\, d \sigma_\xi -\frac{1}{\pi} \int_{\Xi^{(0)}} \partial_t {{N}}_{\alpha -1} \, d\xi   - \frac{1}{\pi}  \sum_{i=1}^{3}\int_{\Xi^{(i)}} \partial_t{\widetilde{N}}_{\alpha -1}\, d\xi \notag
 \\
& \ +\sum_{i=1}^{3} h_i^2 \partial_{x_i}w^{(i)}_{\alpha -1}(0,t)
\Big(1 - \mathrm{v}_i  \int_{2\ell_0}^{3\ell_0} \xi_i\,  \chi^\prime_{\ell_0}(\xi_i) \, d\xi_i\Big).
\end{align}
The  coefficients $\{w^{(i)}_{\alpha}\}_{i=1}^3$ form a  solution to the problem
\begin{equation}\label{limit_al}
\left\{\begin{array}{l}
   \partial_t{w}^{(i)}_{\alpha}(x_i,t) \, + \, \big( v_i^{(i)}(x_i)\, w^{(i)}_{\alpha}(x_i,t) \big)^\prime
  =  \big( w^{(i)}_{\alpha -1}(x_i,t)\big)^{\prime\prime}, \quad (x_i, t) \in (0, \ell_i)\times (0, T), \ \ i\in \{1, 2, 3\},
\\[2mm]
 \sum_{i=1}^3  \mathrm{v}_i  \, h_i^2\, w_{\alpha}^{(i)} (0,t) = \boldsymbol{d_{\alpha}}(t)  \ \ \text{for any} \ \ t \in [0, T],
 \\[2mm]
 w^{(1)}_{\alpha}(\ell_1,t) = 0\ \ \text{for any} \ \ t \in [0, T], \qquad {w}^{(i)}_{\alpha}(x_i,0) =0 \ \ \text{for any} \ \ x_i\in [0, \ell_i], \ \ i\in \{1, 2, 3\}.
 \end{array}\right.
 \end{equation}
As before, $w^{(1)}_{\alpha}$ is a solution to the corresponding hyperbolic mixed problem. But now, since $ \boldsymbol{d_{\alpha}}\neq0,$ we cannot additionally satisfy the continuity condition at the vertex, as for the problems \eqref{limit_al-1} and \eqref{limit_0}. Following approach of \cite{Mel-Roh_preprint-2023}, we propose  the weighted incoming concentration average in the boundary conditions
for $w^{(2)}_{\alpha}$ and $w^{(3)}_{\alpha}$, namely
\begin{equation}\label{bc_2_3}
  w_{\alpha}^{(i)}(0,  t)  =    \dfrac{1}{2\, \mathrm{v}_i \, h_i^2 }  \Big(\boldsymbol{d_{\alpha}}(t) -  \mathrm{v}_1 \, h_1^2 \,  w_{\alpha}^{(1)}(0,  t)\Big), \ \  t \in [0, T], \quad i\in \{2, 3\}.
\end{equation}
Taking into account the assumptions made in Remark~\ref{additional},
 we check the validity of the classical solvability criteria $(\boldsymbol{d_{\alpha}}(0)=\boldsymbol{d_{\alpha}}'(0)=0)$ and the infinitely differentiability of solutions.

Thus, the problem \eqref{limit_al} has a classical solution. This means that the solvability condition both for the problem to define $\widetilde{N}_{\alpha }$ (the problem \eqref{N_p_alpha +k_prob+} for  $p=1$ and $k=0)$ and for the Neumann  problem  \eqref{eq_2+}, \eqref{bc_2} and \eqref{uniq_1} $(p=k=1)$ to define $u^{(i)}_{\alpha+1}$  is satisfied.
 In addition, since $\partial_t {\widetilde{N}}_{\alpha-1}$ uniformly in $t\in [0, T]$  decreases exponentially to zero (see \eqref{rem_exp-decrease+0})
and the other summands in  the right-hand side
\begin{align*}
F_{\alpha} - \partial_t{{N}}_{p\alpha -1} = &\  - \partial_t{\widetilde{N}}_{\alpha -1}(\xi,t) + w^{(i)}_{\alpha}(0,t) \,\chi''_{\ell_0}(\xi_i) - \mathrm{v}_i \, w^{(i)}_{\alpha}(0,t) \,\chi'_{\ell_0}(\xi_i)
  \\
  &\
  +  \partial_{x_i}w^{(i)}_{\alpha -1}(0,t)\, \big(1  - \mathrm{v}_i \, \xi_i \big)\, \chi'_{\ell_0}(\xi_i) + \partial_{x_i}w^{(i)}_{\alpha -1}(0,t)\,  \big(\xi_i \, \chi'_{\ell_0}(\xi_i)\big)'
\end{align*}
in the problem \eqref{N_p_alpha +k_prob+}  $(p=1, \ k=0)$ are uniformly bounded with respect to $(\xi, t)\in \Xi\times [0, T]$ and have the compact supports,  the solution $N_\alpha$  to the problem \eqref{N_p_alpha +k_prob}  has  the following  asymptotics uniform with respect to $t\in [0, T]$:
\begin{equation}\label{rem_exp-decrease+1}
N_\alpha(\xi,t) = w^{(i)}_{\alpha}(0,t)  +  \xi_i \, \partial_{x_i}w^{(i)}_{\alpha -1}(0,t) + \mathcal{ O}(\exp(-\beta_0\xi_i))
\quad \mbox{as} \ \ \xi_i\to+\infty,  \ \  \xi  \in \Xi^{(i)},   \quad i\in \{1, 2, 3\}.
\end{equation}
It is easy to verify that
\begin{equation*}
   N_{\alpha }\big|_{t=0} \equiv  \partial_t N_{\alpha }\big|_{t=0} \equiv  \partial^2_t N_{\alpha }\big|_{t=0} \equiv 0, \quad
  u_{\alpha +1}^{(i)}\big|_{t=0} \equiv 0, \quad i\in \{1,2,3\}.
\end{equation*}

Next, assuming that all  coefficients
$$
\big\{\{w^{(i)}_{p\alpha +m}\}_{i=1}^3, \ N_{p\alpha +m}\big\}_{p\in \{0,1\},\  m\in \{-1,0,\ldots, k-1\}}\quad \text{and} \quad   \big\{u^{(i)}_{p\alpha +m}\big\}_{i\in \{1,2,3\}, \ p\in \{0,1\},\  m\in \{-1,0,\ldots, k\}}
$$   are determined and  that they and  their derivatives in $t$   vanish at $t=0,$ we first determine the value $\boldsymbol{d_{p\alpha +k}}(t)$ by \eqref{const_d*k-p*alpfa}, then the coefficients
\begin{itemize}
  \item $\{w^{(i)}_{p\alpha +k}\}_{i=1}^3$ as a solution to the problem
  \begin{equation}\label{limit_al+k}
\left\{\begin{array}{l}
\partial_t {w}_{p\alpha +k}^{(i)} + \big(  v^{(i)}_i(x_i) \,  w^{(i)}_{p\alpha +k} \big)^\prime
  = \big( w^{(i)}_{p\alpha +k-1}\big)^{\prime\prime},  \quad (x_i, t) \in (0, \ell_i)\times (0, T), \ \ i\in \{1, 2, 3\},
\\[2mm]
 \sum_{i=1}^3  \mathrm{v}_i  \, h_i^2\, w_{p \alpha+ k }^{(i)} (0,t) = \boldsymbol{d_{p\alpha+k}}(t)  \ \ \text{for any} \ \ t \in [0, T],
 \\[2mm]
 w^{(1)}_{p\alpha+k}(\ell_1,t) = 0\ \ \text{for any} \ \ t \in [0, T], \qquad {w}^{(i)}_{p\alpha+k}(x_i,0) =0 \ \ \text{for any} \ \ x_i\in [0, \ell_i], \ \ i\in \{1, 2, 3\};
 \end{array}\right.
 \end{equation}
  \item then for each $i\in \{1,2,3\}$ the coefficient $u^{(i)}_{p\alpha + k+1}$ as a solution to the corresponding Neumann problem
  \eqref{eq_2+}, \eqref{bc_2}, \eqref{uniq_1};
  \item and finally, the coefficient $N_{p\alpha +k}$ as a solution to the problem \eqref{N_p_alpha +k_prob}, which has
  the following  asymptotics uniform with respect to $t\in [0, T]$:
\begin{equation}\label{rem_exp-decrease}
N_{p\alpha +k}(\xi,t) = w^{(i)}_{p\alpha +k}(0,t) + \Psi^{(i)}_{p\alpha +k}(\xi_i,t)  +  \mathcal{ O}(\exp(-\beta_0\, \xi_i))
\quad \mbox{as} \ \ \xi_i\to+\infty,  \ \  \xi  \in \Xi^{(i)},
\end{equation}
where $\Psi^{(i)}_{p\alpha +k}$ is defined in \eqref{Psi},  $i=\{1,2,3\},$ and $\beta_0 >0.$
\end{itemize}
In addition, it is easy to verify that all these coefficients  vanish at $t=0.$

\medskip

Unfortunately, the regular ansatzes  $\mathcal{U}_{\varepsilon}^{(2)}$ and $\mathcal{U}_{\varepsilon}^{(3)}$
don't satisfy the boundary conditions at  the bases $\Upsilon_{\varepsilon}^{(2)} (\ell_2)$ and $\Upsilon_{\varepsilon}^{(3)} (\ell_3),$ respectively. Therefore, we must run the boundary-layer parts \eqref{exp-EVl+alfa} of the asymptotics compensating the residuals of the regular one at $\Upsilon_{\varepsilon}^{(2)} (\ell_2)$ and $\Upsilon_{\varepsilon}^{(3)} (\ell_3)$.

It is additionally assumed that the component $v_i^{(i)}$ of the vector-valued function $\overrightarrow{V_\varepsilon}^{(i)}$  is independent of the variable $x_i$ in a neighborhood of $\Upsilon_{\varepsilon}^{(i)}(\ell_i)$  $(i\in\{2, 3\})$.   This is a technical assumption. In the general case, the function  $v_i^{(i)}$ must be expanded in a Taylor series in a neighborhood of the point $\ell_i.$

Substituting  \eqref{exp-EVl+alfa}  into the differential equation and boundary conditions  of the problem \eqref{probl}  in
 a neighborhood the base $\Upsilon_{\varepsilon}^{(i)} (\ell_i)$ of the thin cylinder $\Omega_{\varepsilon}^{(i)}$ and collecting  coefficients at the same powers of~$\varepsilon$, we get the following  problems:
\begin{equation}\label{prim+probl+k}
 \left\{\begin{array}{rcll}
    \Delta_\eta \Pi_{p\alpha +k}^{(i)}(\eta,t) +  v_i^{(i)}(\ell_i) \partial_{\eta_i}\Pi_{p\alpha +k}^{(i)}(\eta,t)
  & =    & \partial_{t}\Pi_{p\alpha +k-1}^{(i)}(\eta,t),
   &  \eta\in \mathfrak{C}_+^{(i)},
   \\[2mm]
  \partial_{\nu_{\overline{\eta}_1}} \Pi_{p\alpha +k}^{(i)}(\eta,t) & =
   & 0,
   &  \eta\in \partial\mathfrak{C}_+^{(i)} \setminus \Upsilon^{(i)},
   \\[2mm]
  \Pi_{p\alpha +k}^{(i)}\big|_{\eta_i = 0} & =
   & \Phi^{(i)}_{p\alpha +k}(t) ,
   &  \overline{\eta}_i\in\Upsilon^{(i)},
   \\[2mm]
  \Pi_{p\alpha +k}^{(i)}(\eta,t) & \to
   & 0 \ \ \text{as} \ \ \eta_i\to+\infty,
   &  \end{array}\right.
\end{equation}
for $k\in \Bbb N_0 \cup \{-1\}, \ p\in \{0, 1\},\ i\in \{2, 3\}.$ In this sequence of problems,  $\eta=(\eta_1, \eta_2, \eta_3),$ $\eta_i = \frac{\ell_i - x_i}{\varepsilon},$ $\overline{\eta}_i=\frac{\overline{x}_i}{\varepsilon},$
\begin{gather*}
\Upsilon^{(i)}:=\big\{\overline{\eta}_i \in \Bbb R^2\colon |\overline{\eta}_i | < h_i\big\}, \qquad
\mathfrak{C}_+^{(i)}:=\big\{\eta \colon \  \overline{\eta}_i  \in\Upsilon^{(i)}, \quad \eta_i\in(0,+\infty)\big\},
   \\
  \Phi^{(i)}_{p\alpha +k}(t) := \delta_{p\alpha +k, 0} \,  q_{i}(t) - w_{p\alpha +k}^{(i)}(\ell_i,t), \qquad \Pi_{-2}^{(i)} \equiv \Pi_{-1}^{(i)} \equiv
  \Pi_{\alpha-2}^{(i)} \equiv 0.
\end{gather*}

Using the Fourier method, we find solutions to these problems step by step, e.g.,
\begin{gather}\label{Pi_0}
\Pi_{\alpha-1}^{(i)}(\eta_i,t) = \Phi_{\alpha-1}(t) \, e^{- v_i^{(i)}(\ell_i) \, \eta_i}, \qquad \Pi_0^{(i)}(\eta_i,t) =
\Phi_0(t) \, e^{- v_i^{(i)}(\ell_i) \, \eta_i} , \notag
\\
\Pi_{\alpha}^{(i)}(\eta_1,t)  = \left(\Phi_{\alpha}(t) - \frac{\partial_t {\Phi}_{\alpha-1}(t)}{v_i^{(i)}(\ell_3)}\, \eta_1 \right)   e^{-v_i^{(i)}(\ell_i) \, \eta_i}
\end{gather}
Since $v_i^{(i)}(\ell_i)  > 0$  and the factor near $e^{-v_i^{(i)}\, \eta_i}$ are bounded with respect to $t\in [0,T],$
\begin{equation}\label{exp_decay}
  \Pi_{p \alpha +k}^{(i)}(\eta_1,t) = \mathcal{O}\big(e^{- \theta_i\, \eta_i}\big) \quad \text{as} \quad \eta_i \to +\infty
\end{equation}
uniformly in  $t\in [0, T],$ where $\theta_i = \frac{v_i^{(i)}(\ell_i)}{2}.$  Obviously, $\Pi_{p \alpha +k}^{(i)}\big|_{t=0} =0.$

Thus, we  can successively determine all coefficients of the ansatzes \eqref{regul+alfa} -- \eqref{exp-EVl+alfa}.

\section{Justification and asymptotic estimates}\label{justification}
With the help of the smooth cut-off functions  $\chi_{\ell_0}$  (see~\eqref{new-solution_p_alpha +k}) and
\begin{equation}\label{cut-off_functions}
\chi_\delta^{(i)} (x_i) =
    \left\{
    \begin{array}{ll}
        1, & \text{if} \ \ x_i \ge \ell_i -  \delta,
    \\
        0, & \text{if} \ \ x_i \le \ell_i - 2\delta,
    \end{array}
    \right.
\quad i \in \{2, 3\},
\end{equation}
where $\delta$ is a sufficiently small fixed positive number such that $\chi_\delta^{(i)}$ vanishes in the support of $\varphi_\varepsilon^{(i)},$
and the ansatzes \eqref{regul+alfa} -- \eqref{exp-EVl+alfa}, we construct the following series in $\Omega_\varepsilon$:
\begin{equation}\label{asymp_expansion}
    \mathfrak{U}^{(\varepsilon)} :=
        \sum\limits_{k=0}^{+\infty} \sum\limits_{p=0}^{1}  \varepsilon^{p\alpha +k -1 }
        \left\{
  \begin{array}{ll}
   \mathcal{U}^{(1)}_{p\alpha +k -1 }(x, t; \varepsilon) & \text{in} \  \  \Omega^{(1)}_{\varepsilon,3\ell_0,\gamma},
\\[3pt]
   \mathcal{U}^{(i)}_{p\alpha +k -1}(x, t; \varepsilon) +  \mathcal{P}^{(i)}_{p\alpha +k -1}(x_i, t; \varepsilon)& \text{in} \ \
    \Omega^{(i)}_{\varepsilon,3\ell_0,\gamma}, \ \ i\in\{2,3\},
    \\[3pt]
    N_{p\alpha +k -1} \big( \tfrac{x}{\varepsilon}, t \big) & \text{in} \ \  \Omega^{(0)}_{\varepsilon, \gamma},
\\[3pt]
\chi_{\ell_0}\big(\frac{x_i}{\varepsilon^\gamma}\big)\, w_{p\alpha +k -1}^{(i)} (x_i) +
\Big(1- \chi_{\ell_0}\big(\frac{x_i}{\varepsilon^\gamma}\big)\Big) N_{p\alpha +k -1} & \text{in} \ \
\Omega^{(i)}_{\varepsilon,2\ell_0,3\ell_0,\gamma }, \ \ i\in\{1,2,3\},
  \end{array}
\right.
\end{equation}
where $\gamma$ is a fixed number from the interval $(\frac23, 1),$
\begin{equation}\label{U-P}
     \mathcal{U}^{(i)}_{p\alpha +k -1}(x, t; \varepsilon) :=
   w_{p\alpha +k -1}^{(i)} (x_i) + u_{p\alpha +k -1}^{(i)} \big( x_i, \tfrac{\overline{x}_i}{\varepsilon}, t\big),
   \quad
    \mathcal{P}^{(i)}_{p\alpha +k -1}(x_i, t; \varepsilon)  := \chi_\delta^{(i)}(x_i) \,
       \Pi_{p \alpha +k}^{(i)}\big(\tfrac{\ell_i - x_i}{\varepsilon}, t \big),
\end{equation}
 and the parts of the thin graph-like junction $\Omega_\varepsilon$ are defined as follows
\begin{gather*}
\Omega^{(i)}_{\varepsilon,3\ell_0,\gamma}  := \Omega^{(i)}_\varepsilon \, \cap\,  \big\{x\colon x_i \in [3\ell_0 \varepsilon^\gamma, \ell_i]\big\}, \qquad
 \Omega^{(i)}_{\varepsilon,2\ell_0,3\ell_0,\gamma }  := \Omega^{(i)}_\varepsilon \, \cap \, \big\{x\colon x_i \in [2\ell_0 \varepsilon^\gamma, 3\ell_0 \varepsilon^\gamma]\big\}, \
\\[2pt]
\Omega^{(0)}_{\varepsilon, \gamma}:=\Omega^{(0)}_\varepsilon \bigcup \Big(\bigcup\nolimits_{i=1}^3 \Omega^{(i)}_\varepsilon \cap \{x\colon x_i\in [\varepsilon \ell_0,  2\ell_0 \varepsilon^\gamma]\}\Big).
\end{gather*}

Take any $M \in \Bbb N$ and $M >  \frac{3}{2}(1 - \lfloor \alpha\rfloor)$ and  denote by $\mathfrak{U}_{M}^{(\varepsilon)}$ the partial sum of $(\ref{asymp_expansion})$ (see Definition~\ref{Puiseux}).  Based on the properties of the coefficients of the series \eqref{regul+alfa} -- \eqref{exp-EVl+alfa} proved above, we have
 $$
 \mathfrak{U}_{M}^{(\varepsilon)}\big|_{t=0} = 0, \qquad \mathfrak{U}_{M}^{(\varepsilon)}\big|_{x_i=\ell_i} = q_i(t), \quad i\in \{1, 2, 3\}.
 $$

In $\Omega^{(1)}_{\varepsilon,3\ell_0,\gamma}$
\begin{align*}
  \mathfrak{U}_{M}^{(\varepsilon)}(x,t) = &  \
  \sum\limits_{k=0}^{M}  \varepsilon^{\alpha +k -1} \, \Big(w_{\alpha +k -1}^{(1)} (x_1,t) + u_{\alpha +k -1}^{(1)} \big( x_1, \tfrac{\overline{x}_1}{\varepsilon}, t\big)\Big)
  \\
   & \ +  \sum\limits_{k= 1}^{M +\lfloor \alpha\rfloor}   \varepsilon^{k -1} \, \Big(w_{k -1}^{(1)} (x_1,t) + u_{k -1}^{(1)} \big( x_1, \tfrac{\overline{x}_1}{\varepsilon}, t\big)\Big) ,
\end{align*}
and due to \eqref{lim_2} and \eqref{eq_2+} this partial sum satisfies the differential equation
\begin{equation}\label{Res_1}
\partial_t \mathfrak{U}_{M}^{(\varepsilon)} -  \varepsilon\, \Delta_x \mathfrak{U}_{M}^{(\varepsilon)} +
  \mathrm{div}_x \big( \overrightarrow{V_\varepsilon} \, \mathfrak{U}_{M}^{(\varepsilon)}\big)
=  \varepsilon^{M +\lfloor \alpha\rfloor -1}\, \mathcal{R}^{(1)}_{M +\lfloor \alpha\rfloor -1 } +  \varepsilon^{\alpha +M -1 }\, \mathcal{R}^{(1)}_{\alpha +M -1 } \quad \text{in} \ \ \Omega^{(1)}_{\varepsilon, \gamma}\times (0,T)
 \end{equation}
where
\begin{multline}\label{Res_1+}
\mathcal{R}^{(1)}_{\mathfrak{p} -1 }(x_1,\bar{\xi}_1, t) =
\Big(  v^{(1)}_1(x_1) \,  u^{(1)}_{\mathfrak{p} - 1}(x_1,\bar{\xi}_1, t)  \Big)^\prime
  +    \partial_t {u}_{\mathfrak{p}-1}^{(1)}(x_1,\bar{\xi}_1, t) -  \Big( u^{(1)}_{\mathfrak{p}-2}(x_1,\bar{\xi}_1, t) \Big)^{\prime\prime}
    \\
    +   \mathrm{div}_{\bar{\xi}_1} \Big( \overline{V}^{(1)}(x_1, \bar{\xi}_1) \,
            \big[ w^{(1)}_{\mathfrak{p}-1}(x_1, t) + u^{(1)}_{\mathfrak{p}-1}(x_1,\bar{\xi}_1, t) \big]\Big)
            -  \varepsilon \Big( u^{(1)}_{\mathfrak{p}-1}(x_1,\bar{\xi}_1, t) \Big)^{\prime\prime}
\end{multline}
for $\mathfrak{p} \in \{M +\lfloor \alpha\rfloor, \  \alpha +M \}.$
Thanks to our assumptions,
\begin{equation}\label{Res_2-}
  \sup_{\Omega^{(i)} _{\varepsilon,\gamma}\times (0,T)} |\mathcal{R}^{(1)}_{\mathfrak{p} -1 }(x_1,\tfrac{\bar{x}_1}{\varepsilon}, t)| \le C^{(1)}_M,
\end{equation}
where the constant $C^{(1)}_M$ is independent of $\varepsilon.$
\begin{remark}\label{rem_const}
Hereinafter, all constants in inequalities are independent of the parameter~$\varepsilon.$
\end{remark}

In $\Omega^{(2)}_{\varepsilon,3\ell_0,\gamma}$ and $\Omega^{(3)}_{\varepsilon,3\ell_0,\gamma}$ the partial sum $\mathfrak{U}_{M}^{(\varepsilon)}$
additionally contains the partial sum of the  boundary-layer ansatz (see \eqref{asymp_expansion}). Therefore, residuals from this partial sum in the corresponding differential equation are the sum of the terms $\sum_{\mathfrak{p} \in \{M +\lfloor \alpha\rfloor, \  \alpha +M \}}  \varepsilon^{\mathfrak{p} -1 } \mathcal{R}^{(i)}_{\mathfrak{p} +M -1 },$ where $\mathcal{R}^{(i)}_{\mathfrak{p} +M -1 }$ is estimated in the same way as in \eqref{Res_2-} $(i\in \{2,3\}),$ and
\begin{multline*}
\chi_\delta^{(i)} (x_i)\sum_{\mathfrak{p} \in \{M +\lfloor \alpha\rfloor, \, \alpha +M \}}  \varepsilon^{\mathfrak{p} -1 } \,  \partial_t \Pi_{\mathfrak{p} -1}^{(i)}
   \\
  +\sum\limits_{k=0}^{M}  \varepsilon^{\alpha +k -1 } \Big( 2 \,  \big(\chi_\delta^{(i)}\big)'  \partial_{\xi_i} \Pi_{\alpha +k -1 }^{(i)}   + v_i^{(i)}(\ell_i)  \,  \big(\chi_\delta^{(i)}\big)'  \,  \Pi_{\alpha +k -1 }^{(i)} - \varepsilon \, \big(\chi_\delta^{(3)}\big)^{\prime\prime}  \Pi_{\alpha +k -1 }^{(i)}\Big)
  \\
  + \sum\limits_{k= 1}^{M +\lfloor \alpha\rfloor}   \varepsilon^{k -1} \Big( 2 \,  \big(\chi_\delta^{(i)}\big)'  \partial_{\xi_i} \Pi_{k -1 }^{(i)}   + v_i^{(i)}(\ell_i)  \,  \big(\chi_\delta^{(i)}\big)'  \,  \Pi_{k -1 }^{(i)} - \varepsilon \, \big(\chi_\delta^{(3)}\big)^{\prime\prime}  \Pi_{k -1 }^{(i)}\Big).
\end{multline*}
The supports of summands in the second and third  lines coincide with $\mathrm{supp}\big(\big(\chi_\delta^{(i)}\big)'\big),$ where
the functions $\{\Pi_{p\alpha +k -1 }^{(i)}\}$ exponentially small as $\varepsilon$ tends to zero (see \eqref{exp_decay}). Therefore, residuals from $\mathfrak{U}_{M}^{(\varepsilon)}$ in the differential equation in $\Omega^{(i)}_{\varepsilon,3\ell_0,\gamma}$ are also of  order
$\mathcal{O}\big(\varepsilon^{M +\lfloor \alpha\rfloor -1}\big) + \mathcal{O}\big(\varepsilon^{\alpha +M -1 }\big)$
for sufficiently small $\varepsilon.$

Using \eqref{bc_2} and taking the zero Neumann boundary condition for the solutions $\{\Pi_{p\alpha +k -1 }^{(i)}\}$ (see \eqref{prim+probl+k}) into account, we have
\begin{equation}\label{Res_3+}
-  \varepsilon \,  \partial_{\boldsymbol{\nu}_\varepsilon} \mathfrak{U}_{M}^{(\varepsilon)} +  \mathfrak{U}_{M}^{(\varepsilon)} \, \overrightarrow{V_\varepsilon}\boldsymbol{\cdot}\boldsymbol{\nu}_\varepsilon = \varepsilon^{\alpha}  \varphi^{(i)}_\varepsilon + \sum_{\mathfrak{p} \in \{M +\lfloor \alpha\rfloor, \,  \alpha +M \}} \varepsilon^{\mathfrak{p}}\,  \Phi_{\mathfrak{p}}^{(i)} \quad \text{on} \ \
  \Gamma^{(i)}_{\varepsilon, \gamma}\times (0, T),
\end{equation}
where  the lateral surfaces $\Gamma^{(i)}_{\varepsilon, \gamma} := \Gamma^{(i)} _\varepsilon \cap \{x\colon x_i \in [3 \ell_0 \varepsilon^\gamma, \ell_i) \},$ $i\in \{1, 2, 3\},$
\begin{equation}\label{Phi_p}
\Phi_{\mathfrak{p}}^{(i)}(x_i,\bar{\xi}_i, t) := \big( w^{(i)}_{\mathfrak{p} -1} + u^{(i)}_{\mathfrak{p} -1} \big) \, \overline{V}^{(i)} \boldsymbol{\boldsymbol{\cdot}} \bar{\nu}_{\bar{\xi}_i},
\end{equation}
and  there are positive constants $\varepsilon_0$ and $\tilde{C}^{(i)}_M$ such that for all $\varepsilon \in (0, \varepsilon_0)$
\begin{equation}\label{Res_5}
  \sup_{\Gamma^{(i)}_{\varepsilon, \gamma} \times (0,T)} |\Phi_{\mathfrak{p} +M}^{(i)}(x_i,\tfrac{\bar{x}_i}{\varepsilon}, t)| \le \tilde{C}^{(i)}_M.
\end{equation}
In addition, the functions $\{\Phi_{\mathfrak{p}}^{(i)}\}_{i=1}^3$ vanish at circular strips on the lateral surfaces of the thin cylinders near their bases $\{\Upsilon_{\varepsilon}^{(i)} (\ell_i)\}_{i=1}^3$, since the functions  $\{\overline{V}^{(i)}\}$  and  $\{\varphi_\varepsilon^{(i)}\}$ vanish there.

In virtue of \eqref{N_p_alpha +k_prob}, we get
\begin{equation}\label{Res_2}
\partial_t \mathfrak{U}_{M}^{(\varepsilon)} -  \varepsilon\, \Delta_x \mathfrak{U}_{M}^{(\varepsilon)} +
  \mathrm{div}_x \big( \overrightarrow{V_\varepsilon} \, \mathfrak{U}_{M}^{(\varepsilon)}\big)
=  -\sum_{\mathfrak{p} \in \{M +\lfloor \alpha\rfloor, \,  \alpha +M \}}   \varepsilon^{\mathfrak{p} -1 } \partial_t {{N}}_{\mathfrak{p} -1} \quad \text{in} \ \ \Omega^{(0)}_{\varepsilon, \gamma} \times (0,T),
 \end{equation}
and
\begin{equation}\label{res_bc}
-  \varepsilon \,  \partial_{\boldsymbol{\nu}_\varepsilon} \mathfrak{U}_{M}^{(\varepsilon)} =  \varepsilon^{\alpha}  \varphi^{(0)}_\varepsilon
\quad
   \text{on} \ \Big(\partial\Omega^{(0)}_{\varepsilon, \gamma} \setminus \Big\{
 \bigcup_{i=1}^3 \overline{\Upsilon_\varepsilon^{(i)} (2\ell_0 \varepsilon^\gamma)}
\Big\} \Big) \times (0, T).
\end{equation}

Now it remains to calculate and estimate residuals left by $\mathfrak{U}_{M}^{(\varepsilon)}$  in the differential equations in the
thin and small cylinders $\Omega^{(i)}_{\varepsilon,2\ell_0,3\ell_0,\gamma },$ $i\in \{1, 2, 3\},$
and in the boundary conditions on their  lateral surfaces. Since $\varphi^{(i)}_\varepsilon,\,  \overline{V}^{(i)}$ vanish there,
\begin{equation}\label{Res_8}
-  \varepsilon \,  \partial_{\boldsymbol{\nu}_\varepsilon} u_\varepsilon +  u_\varepsilon \, \overrightarrow{V_\varepsilon}\boldsymbol{\cdot}\boldsymbol{\nu}_\varepsilon =  0
 \end{equation}
on the corresponding lateral surface of $\Omega^{(i)}_{\varepsilon,2\ell_0,3\ell_0,\gamma }.$
Similar to \eqref{Res_1} and \eqref{Res_2}, but now taking into account Remark~\ref{r_2_1}, we find
\begin{multline}\label{Res_3}
\partial_t \mathfrak{U}_{M}^{(\varepsilon)} -  \varepsilon\, \Delta_x \mathfrak{U}_{M}^{(\varepsilon)} +
  \mathrm{div}_x \big( \overrightarrow{V_\varepsilon} \, \mathfrak{U}_{M}^{(\varepsilon)}\big)
=  -
\sum_{\mathfrak{p} \in \{M +\lfloor \alpha\rfloor, \,  \alpha +M \}}   \varepsilon^{\mathfrak{p} -1 }
\Big[\varepsilon \chi_{\ell_0}\big(\frac{x_i}{\varepsilon^\gamma}\big) \big( w^{(i)}_{\mathfrak{p} -1}(x_i,t)\big)^{\prime\prime}
+ \big(1- \chi_{\ell_0}\big(\frac{x_i}{\varepsilon^\gamma}\big)\big) \partial_t {{N}}_{\mathfrak{p} -1}\Big]
\\
- \chi''_{\ell_0}\big(\frac{x_i}{\varepsilon^\gamma}\big) \sum\limits_{k=1}^{M +\lfloor \alpha\rfloor}  \varepsilon^{k -2 \gamma}
\Big(w_{k -1}^{(i)} (x_i) - N_{k -1}\Big)  + \mathrm{v}_i \, \chi'_{\ell_0}\big(\frac{x_i}{\varepsilon^\gamma}\big) \sum\limits_{k=1}^{M +\lfloor \alpha\rfloor}  \varepsilon^{k -1-\gamma} \Big(w_{k -1}^{(i)} (x_i) - N_{k -1}\Big)
\\
- 2 \chi'_{\ell_0}\big(\frac{x_i}{\varepsilon^\gamma}\big) \sum\limits_{k=1}^{M +\lfloor \alpha\rfloor}  \varepsilon^{k -\gamma}
\Big(\partial_{x_i}w_{k -1}^{(i)} (x_i) - \varepsilon^{-1} \partial_{\xi_i} N_{k -1}\Big)
\\
- \chi''_{\ell_0}\big(\frac{x_i}{\varepsilon^\gamma}\big) \sum\limits_{k=0}^{M}  \varepsilon^{\alpha +k -2 \gamma}
\Big(w_{\alpha +k -1}^{(i)} (x_i) - N_{\alpha +k -1}\Big) + \mathrm{v}_i \, \chi'_{\ell_0}\big(\frac{x_i}{\varepsilon^\gamma}\big) \sum\limits_{k=0}^{M}   \varepsilon^{\alpha +k -1-\gamma}
\Big(w_{\alpha +k -1}^{(i)} (x_i) - N_{\alpha +k -1}\Big)
\\
- 2 \chi_{\ell_0}'\big(\frac{x_i}{\varepsilon^\gamma}\big) \sum\limits_{k=0}^{M}  \varepsilon^{\alpha +k -\gamma}
\Big(\partial_{x_i}w_{\alpha +k -1}^{(i)} (x_i) - \varepsilon^{-1} \partial_{\xi_i} N_{\alpha +k -1}\Big)
 \quad \text{in} \ \
\Omega^{(i)}_{\varepsilon,2\ell_0,3\ell_0,\gamma }.
 \end{multline}
The terms in the first line of the right-hand side of \eqref{Res_3} are of order
$\mathcal{O}\big(\varepsilon^{M +\lfloor \alpha\rfloor -1}\big) + \mathcal{O}\big(\varepsilon^{\alpha +M -1 }\big).$
 The rest of the terms are localized in the support of $\chi'_{\ell_0}\big(\frac{x_i}{\varepsilon^\gamma}\big)$. Therefore, using the Taylor formula for the functions $\{w_{p\alpha +k -1}^{(i)}\}$ at the point $x_i=0$ and the formula \eqref{new-solution_p_alpha +k}, the summands  in the other lines of \eqref{Res_3} can be rewritten as follows
\begin{gather*}
\chi''_{\ell_0}\big(\tfrac{x_i}{\varepsilon^\gamma}\big) \sum\limits_{k=1}^{ M +\lfloor \alpha\rfloor}  \varepsilon^{k -2 \gamma}
\widetilde{N}_{k -1} + \mathcal{O}\big(\varepsilon^{\gamma(M +\lfloor \alpha\rfloor - 2)+1}\big)
 - \mathrm{v}_i \, \chi'_{\ell_0}\big(\tfrac{x_i}{\varepsilon^\gamma}\big) \sum\limits_{k=1}^{M +\lfloor \alpha\rfloor}   \varepsilon^{k -1-\gamma} \widetilde{N}_{k -1} + \mathcal{O}\big(\varepsilon^{\gamma (M +\lfloor \alpha\rfloor -1)}\big)
\\
   + 2 \chi'_{\ell_0}\big(\tfrac{x_i}{\varepsilon^\gamma}\big) \sum\limits_{k=0}^{M +\lfloor \alpha\rfloor - 1} \varepsilon^{k-\gamma} \partial_{\xi_i}\widetilde{N}_{k}
  + \mathcal{O}\big(\varepsilon^{\gamma (M +\lfloor \alpha\rfloor -2)+1}\big)
 \\
\chi''_{\ell_0}\big(\tfrac{x_i}{\varepsilon^\gamma}\big) \sum\limits_{k=0}^{M}  \varepsilon^{\alpha +k -2 \gamma}
\widetilde{N}_{\alpha +k -1} +  \mathcal{O}\Big(\varepsilon^{\alpha + \gamma M -\gamma}\Big)
 - \mathrm{v}_i \, \chi'_{\ell_0}\big(\tfrac{x_i}{\varepsilon^\gamma}\big) \sum\limits_{k=0}^{M}  \varepsilon^{\alpha +k -1-\gamma} \widetilde{N}_{\alpha +k -1} + \mathcal{O}\Big(\varepsilon^{\alpha + \gamma M -1}\Big).
 \\
  + 2 \chi'_{\ell_0}\big(\tfrac{x_i}{\varepsilon^\gamma}\big) \sum\limits_{k=-1}^{M-1} \varepsilon^{k-\gamma} \partial_{\xi_i}\widetilde{N}_{\alpha +k}
  +  \mathcal{O}\Big(\varepsilon^{\alpha + \gamma M -\gamma}\Big).
 \end{gather*}
Taking into account  \eqref{new-solution_p_alpha +k} and \eqref{rem_exp-decrease}, the  maximum of $|\widetilde{N}_\mathfrak{m}|$ and $|\partial_{\xi_i}\widetilde{N}_\mathfrak{m}|$ over
$$
 \big(\Omega_\varepsilon^{(i)} \cap \big\{ x: \  x_i\in  [2\ell_0\varepsilon^\gamma, 3\ell_0\varepsilon^\gamma] \big\} \big) \times [0, T]
 $$
 are of order $\exp\big(-\beta_0 2 \ell_0 \, \varepsilon^{\gamma -1}\big),$ i.e.,
these terms exponentially decrease as the parameter $\varepsilon$ tends to zero.  Thus,  the right-hand side of \eqref{Res_3} is of order
$\mathcal{O}\big(\varepsilon^{\gamma (M +\lfloor \alpha\rfloor -1)}\big) + \mathcal{O}\big(\varepsilon^{\alpha + \gamma M -1}\big),$
 and these values are  infinitesimal as $\varepsilon \to 0,$ since $M \in \Bbb N,$ $ M >   \frac{3}{2}(1 - \lfloor \alpha\rfloor),$ and
  $\gamma\in (\frac23, 1).$ In addition, if $\gamma > 1 - \frac{\alpha - \lfloor \alpha\rfloor }{1 -\lfloor \alpha\rfloor},$ then
$\gamma (M +\lfloor \alpha\rfloor -1) < \alpha + \gamma M -1,$ and therefore,
$\mathcal{O}\big(\varepsilon^{\gamma (M +\lfloor \alpha\rfloor -1)}\big) + \mathcal{O}\big(\varepsilon^{\alpha + \gamma M -1}\big)=
\mathcal{O}\big(\varepsilon^{\gamma (M +\lfloor \alpha\rfloor -1)}\big).$
In what follows, we consider  the parameter
\begin{equation}\label{gamma_cond}
 \gamma \in \left( \max\Big\{\frac{2}{3} , \  1 - \frac{\alpha - \lfloor \alpha\rfloor }{1 -\lfloor \alpha\rfloor}\Big\} , \ 1 \right).
\end{equation}
It should be noted here that $ 1 - \frac{\alpha - \lfloor \alpha\rfloor }{1 -\lfloor \alpha\rfloor} < 1$ and
$
\lim_{\alpha \to -\infty} \left( 1 - \frac{\alpha - \lfloor \alpha\rfloor }{1 -\lfloor \alpha\rfloor} \right) = 1.
$

\begin{remark}\label{Rem3-6}
As was noted in Remark~\ref{rem_const}, constants in inequalities are independent of $\varepsilon,$ but they depend on the value of the parameters  $M,$ $\alpha,$ $\gamma,$ and $i$. In what follow we indicate only the dependence of $M$ and use the same notation $C_M$ for all constants.
\end{remark}

\begin{theorem}\label{Th_1}
Let  assumptions made in Section~\ref{Sec:Statement},  in  Remark~\ref{additional} and \eqref{gamma_cond}  hold. Then the series \eqref{asymp_expansion} is the asymptotic expansion for the solution $u_\varepsilon$  to the problem~\eqref{probl} in both the Banach space   $C(\overline{\Omega}_\varepsilon\times [0,T])$ and the Sobolev space $L^2\big((0,T); H^1(\Omega_\varepsilon)\big);$ and for  any $M\in \Bbb N$ and  $ M >  \frac{3}{2}(1 - \lfloor \alpha\rfloor)$ there exist $C_M>0$ and $\varepsilon_0>0$ such that for all $\varepsilon\in(0, \varepsilon_0)$ the  estimates
$$
 \left\|u_\varepsilon - \mathfrak{U}_{P}^{(\varepsilon)}\right\|_{C(\overline{\Omega}_\varepsilon\times [0,T])} \le C_M \, \varepsilon^{\gamma (M +\lfloor \alpha\rfloor -1)}
$$
 and
 \begin{equation}\label{main_2+}
 \tfrac{1}{\sqrt{|\Omega_\varepsilon|}}\, \left\|\nabla_x(u_\varepsilon - \mathfrak{U}_{P+1}^{(\varepsilon)})\right\|_{L^2(\Omega_\varepsilon\times (0, T))}
    \le C_M \, \varepsilon^{\gamma (M +\lfloor \alpha\rfloor -1) - \frac12}
 \end{equation}
are satisfied,  where $\mathfrak{U}_{P}^{(\varepsilon)}$ is the partial sum of the series \eqref{asymp_expansion},   $P = \lfloor \gamma (M +\lfloor \alpha\rfloor -1) \rfloor +1 - \lfloor \alpha \rfloor,$ and $|\Omega_\varepsilon|$ is  the Lebesque measure of $\Omega_\varepsilon$.
\end{theorem}
\begin{proof} {\bf 1.} From calculations above it follows that the partial sum $\mathfrak{U}_{M}^{(\varepsilon)}$ leaves  the biggest residuals in  $\Omega^{(i)}_{\varepsilon,2\ell_0,3\ell_0,\gamma },$ $i\in \{1, 2, 3\}.$
Therefore, taking into account  the  estimates of residuals carried out in this section, the difference between the partial sum $\mathfrak{U}_{M}^{(\varepsilon)}$ of (\ref{asymp_expansion}) and the solution to the problem \eqref{probl} satisfies the following relations:
\begin{equation}\label{dif_1}
  \partial_t(\mathfrak{U}_{M}^{(\varepsilon)} - u_\varepsilon)   -  \varepsilon\, \Delta_x( \mathfrak{U}_{M}^{(\varepsilon)} - u_\varepsilon) +
  \mathrm{div}_x \big( \overrightarrow{V_\varepsilon}^{(i)} \, (\mathfrak{U}_{M}^{(\varepsilon)} - u_\varepsilon)\big)
   = \varepsilon^{\gamma (M +\lfloor \alpha\rfloor -1)} \mathcal{R}^{(i)}_{\gamma (M +\lfloor \alpha\rfloor -1)} \quad\text{in}  \ \Omega_\varepsilon^{(i)}\times (0,T),
\end{equation}
$$
-  \varepsilon \,  \partial_{\boldsymbol{\nu}_\varepsilon}(\mathfrak{U}_{M}^{(\varepsilon)} - u_\varepsilon) +  (\mathfrak{U}_{M}^{(\varepsilon)} - u_\varepsilon)\, \overrightarrow{V_\varepsilon}\boldsymbol{\cdot}\boldsymbol{\nu}_\varepsilon  =
\sum_{\mathfrak{p} \in \{M +\lfloor \alpha\rfloor, \,  \alpha +M \}} \varepsilon^{\mathfrak{p}}\,  \Phi_{\mathfrak{p}}^{(i)} \quad \text{on} \ \
  \Gamma^{(i)}_{\varepsilon, \gamma}\times (0, T),
$$
$$
(\mathfrak{U}_{M}^{(\varepsilon)} - u_\varepsilon)\big|_{x_i= \ell_i}
 =  0 \quad  \text{on} \ \Upsilon_{\varepsilon}^{(i)} (\ell_i)\times (0,T),
\ \ i\in\{1,2,3\},
$$
\begin{equation}\label{dif_2}
\partial_t(\mathfrak{U}_{M}^{(\varepsilon)} - u_\varepsilon) -  \varepsilon\, \Delta_x (\mathfrak{U}_{M}^{(\varepsilon)} - u_\varepsilon) +
   \overrightarrow{V_\varepsilon}^{(0)} \cdot \nabla_x(\mathfrak{U}_{M}^{(\varepsilon)} - u_\varepsilon)  =
    -\sum_{\mathfrak{p} \in \{M +\lfloor \alpha\rfloor, \,  \alpha +M \}}   \varepsilon^{\mathfrak{p} -1 } \mathcal{R}^{(0)}_{\mathfrak{p}-1} \quad \text{in} \ \ \Omega^{(0)}_{\varepsilon, \gamma} \times (0,T),
\end{equation}
$$
 -  \varepsilon \,  \partial_{\boldsymbol{\nu}_\varepsilon}(\mathfrak{U}_{M}^{(\varepsilon)} - u_\varepsilon)  = 0  \quad
    \text{on} \ \ \Gamma_\varepsilon^{(0)}\times (0,T),
$$
$$
  (\mathfrak{U}_{M}^{(\varepsilon)} - u_\varepsilon)\big|_{t=0}  =  0 \quad  \text{on} \ \Omega_{\varepsilon},
$$
where $\mathcal{R}^{(0)}_{\mathfrak{p}-1} = - \partial_t {{N}}_{\mathfrak{p} -1}$ for
$\mathfrak{p} \in \{M +\lfloor \alpha\rfloor, \,  \alpha +M \}, $ the residual $\Phi_{\mathfrak{p}}^{(i)}$ is determined in \eqref{Phi_p} and satisfied the estimate \eqref{Res_5},
$$
\sup_{\Omega^{(i)}_{\varepsilon} \times (0,T)} \big|\mathcal{R}^{(i)}_{\gamma (M +\lfloor \alpha\rfloor -1)}(x, t; \varepsilon)\big| \le C_M.
$$

From the maximum principle proved in \cite[Lemma 5.1]{Mel-Roh_preprint-2022} for parabolic problems in thin graph-like junctions,  we obtain
\begin{equation}\label{for_main_1}
\|u_\varepsilon - \mathfrak{U}_{M}^{(\varepsilon)}\|_{C(\overline{\Omega}_\varepsilon\times [0,T])} :=  \max_{\overline{\Omega_\varepsilon}\times [0, T]} |u_\varepsilon - \mathfrak{U}_{M}^{(\varepsilon)}|  \le C_M \,
   \varepsilon^{\gamma (M +\lfloor \alpha\rfloor -1)}.
 \end{equation}
 Since $M$ is an arbitrary natural number and  $ M >  \frac{3}{2}(1 - \lfloor \alpha\rfloor)$, the inequality  \eqref{for_main_1} means, based on Definition~\ref{Puiseux}, that the series \eqref{asymp_expansion} is the asymptotic expansion of the solution $u_\varepsilon$  in   $C(\overline{\Omega}_\varepsilon\times [0,T]).$

 The partial sum $\mathfrak{U}_{M}^{(\varepsilon)}$ contains terms that are infinitesimal with respect to $\varepsilon^{\gamma (M +\lfloor \alpha\rfloor -1)}$ for $\varepsilon \to 0$. Therefore,  from  \eqref{for_main_1} it follows the inequality
  \begin{equation}\label{main_2}
   \left\|u_\varepsilon - \mathfrak{U}_{P}^{(\varepsilon)}\right\|_{C(\overline{\Omega}_\varepsilon\times [0,T])} \le C_M \, \varepsilon^{\gamma (M +\lfloor \alpha\rfloor -1)},
 \end{equation}
 where $P = \lfloor \gamma (M +\lfloor \alpha\rfloor -1) \rfloor +1 - \lfloor \alpha \rfloor$; it is easy to verify that $1 - \lfloor \alpha \rfloor \le P \le M -1.$

\medskip

{\bf 2.} We multiply the differential equations \eqref{dif_1} and \eqref{dif_2}  with  $U_\varepsilon := \mathfrak{U}_{M}^{(\varepsilon)} - u_\varepsilon$ and integrate them over the corresponding domain and over $ (0, \tau),$ where $\tau$ is an arbitrary number from $(0, T).$ Integrating by parts and taking into account the boundary conditions and the initial condition for $U_\varepsilon,$ \eqref{Res_2}, \eqref{Res_5},
and the volume of the domains being integrated over, we derive
 \begin{align}\label{s_1}
  \| {\nabla_x U_\varepsilon}\|^2_{L^2(\Omega_\varepsilon\times (0, \tau))}
\le & \
\frac{1}{\varepsilon} \int\limits_{0}^{\tau} \bigg(\int_{\Omega^{(0)}_\varepsilon} U_\varepsilon\,
    \overrightarrow{V_\varepsilon} \cdot
    {\nabla_x U_\varepsilon} \, dx + \sum_{i=1}^{3}\int_{\Omega^{(i)}_\varepsilon} U_\varepsilon\,
    \overrightarrow{V_\varepsilon} \cdot
    {\nabla_x U_\varepsilon} \, dx\bigg) dt \notag
 \\
 & \ + C_M \,  \big(\varepsilon^{M +\lfloor \alpha\rfloor} + \varepsilon^{\gamma ( M +\lfloor \alpha\rfloor) +1} \big)\, \|U_\varepsilon \|_{C(\overline{\Omega}_\varepsilon\times [0,T])}.
   \end{align}

Owing to the assumptions for the vector field  $\overrightarrow{V_\varepsilon}$, in particular the incompressibleness  of $\overrightarrow{V_\varepsilon}^{(0)}$  in $\Omega^{(0)}_\varepsilon $
and the Dirichlet conditions for $U_\varepsilon$ on $\{\Upsilon_{\varepsilon}^{(i)} (\ell_i)\}_{i=1}^3,$
\begin{align}\label{incom}
\frac{1}{\varepsilon} \Big|\int_{\Omega^{(0)}_\varepsilon} U_\varepsilon\,
    \overrightarrow{V_\varepsilon} \cdot
    {\nabla_x U_\varepsilon} \, dx & \ + \, \sum_{i=1}^{3}\int_{\Omega^{(i)}_\varepsilon} U_\varepsilon\,
    \overrightarrow{V_\varepsilon} \cdot
    {\nabla_x U_\varepsilon} \, dx\Big| \notag
 \\
 &  = \, \frac{1}{\varepsilon}\,  \Big| - \frac{1}{2}\sum_{i=1}^{3}\int_{\Omega^{(i)}_\varepsilon} U^2_\varepsilon  \, \partial_{x_i}v^{(i)}_i  \, dx
 + \varepsilon \sum_{i=1}^{3}\int_{\Omega^{(i)}_\varepsilon} U_\varepsilon\,
    \overline{V}^{(i)} \cdot {\nabla_{\bar{x}_i} U_\varepsilon} \, dx\Big|  \notag
 \\
& \le  \, C  \varepsilon \,  \|U_\varepsilon \|^2_{C(\overline{\Omega}_\varepsilon\times [0,T])} +  \frac{1}{2}  \, \| {\nabla_x U_\varepsilon}\|^2_{L^2(\Omega_\varepsilon\times (0, \tau))}
\end{align}

Using \eqref{for_main_1}, we derive from \eqref{s_1} and \eqref{incom} the inequality
\begin{equation}\label{ineq_1}
\left\| \nabla_x(u_\varepsilon - \mathfrak{U}_{M}^{(\varepsilon)}) \right\|_{L^2(\Omega_\varepsilon\times (0, T))} \le C_M \,
   \varepsilon^{\gamma (M +\lfloor \alpha\rfloor -1) + \frac12} .
\end{equation}

The inequality \eqref{ineq_1} implies that the series \eqref{asymp_expansion} is the  asymptotic expansion of the solution $u_\varepsilon$ in the Sobolev space $L^2\big ((0,T); H^1(\Omega_\varepsilon)\big).$
It should be emphasized that $\|u \|_{L^2(\Omega_\varepsilon)} \le C \|\nabla_x u \|_{L^2(\Omega_\varepsilon)}$
for any function  from the Sobolev space $H^1(\Omega_\varepsilon)$ whose traces
on the bases $\{\Upsilon_{\varepsilon}^{(i)}(\ell_i)\}_{i=1}^3$ are equal to zero (for more detail see \cite[Sec.~2]{M-AA-2021}).
From \eqref{ineq_1}  follows \eqref{main_2+}, with the rescaled $L^2$-norm on the left side.
\end{proof}

For applied problems it is not necessary to construct a complete asymptotic expansion of the solution. It is sufficient to approximate the solution to the required accuracy. Weaker assumptions about the smoothness of the coefficients and the given functions are then required.
For instance, let us consider case when $\alpha \in (0, 1).$ Then, writing \eqref{main_2} and \eqref{main_2+} for $M=2$, we get
 \begin{equation}\label{main_U_1}
   \left\|u_\varepsilon - \mathfrak{U}_{1}^{(\varepsilon)}\right\|_{C(\overline{\Omega}_\varepsilon\times [0,T])} \le C_2 \, \varepsilon^{\gamma}
 \end{equation}
 and
 \begin{equation}\label{main_U_2}
 \tfrac{1}{\sqrt{|\Omega_\varepsilon|}}\, \left\|\nabla_x(u_\varepsilon - \mathfrak{U}_{2}^{(\varepsilon)})\right\|_{L^2(\Omega_\varepsilon\times (0, T))}
    \le C_2 \, \varepsilon^{\gamma - \frac12}
 \end{equation}
 where $\gamma$ is a fixed number from the interval $\left( \max\big\{\frac{2}{3} , \  1 - \alpha \big\} , \ 1 \right).$
 The partial sum \begin{equation}\label{zero_app}
 \mathfrak{U}_{1}^{(\varepsilon)} =
\left\{
  \begin{array}{ll}
   \varepsilon^{\alpha -1} w_{\alpha -1}^{(1)} (x_1, t) + w_0^{(1)} (x_1, t) + \varepsilon^{\alpha} \big(w_{\alpha}^{(1)} (x_1, t) +
   u_{\alpha}^{(1)}( x_1, \tfrac{\overline{x}_1}{\varepsilon}, t)\big) , & x \in \Omega^{(1)}_{\varepsilon,3\ell_0,\gamma},
\\[4pt]
   \varepsilon^{\alpha -1} w_{\alpha -1}^{(i)} (x_i, t) + w_0^{(i)} (x_i, t) + \varepsilon^{\alpha} \big(w_{\alpha}^{(i)} (x_i, t) +
   u_{\alpha}^{(i)}( x_i, \tfrac{\overline{x}_i}{\varepsilon}, t)\big) &
   \\[2pt]
   + \chi_\delta^{(i)} (x_i)\big( \varepsilon^{\alpha -1} \Pi_{\alpha-1}^{(i)}\big(\tfrac{\ell_i - x_i}{\varepsilon}, t \big)   +  \Pi_{0}^{(i)}\big(\tfrac{\ell_i - x_i}{\varepsilon}, t \big) + \varepsilon^{\alpha} \Pi_{\alpha}^{(i)}\big),
    & x \in \Omega^{(i)}_{\varepsilon,3\ell_0,\gamma}, \ i\in\{2, 3\},
 \\[4pt]
 \varepsilon^{\alpha -1} N_{\alpha-1}\big(\frac{x}{\varepsilon}, t\big)  + N_0\big(\frac{x}{\varepsilon}, t\big) + \varepsilon^{\alpha} N_{\alpha}\big(\frac{x}{\varepsilon}, t\big), &x \in \Omega^{(0)}_{\varepsilon, \gamma},
\\[4pt]
\chi_{\ell_0}^{(i)}\left(\frac{x_i}{\varepsilon^\gamma}\right) \big(
\varepsilon^{\alpha -1} w_{\alpha -1}^{(i)} (x_i, t) + w_0^{(i)} (x_i, t) + \varepsilon^{\alpha} w_{\alpha}^{(i)} (x_i, t) \big) &
  \\[2pt]
  +
\big(1- \chi_{\ell_0}^{(i)}\left(\frac{x_i}{\varepsilon^\gamma}\right)\big) \left(
\varepsilon^{\alpha -1} N_{\alpha-1}\big(\frac{x}{\varepsilon}, t\big)  + N_0\big(\frac{x}{\varepsilon}, t\big) + \varepsilon^{\alpha} N_{\alpha}\big(\frac{x}{\varepsilon}, t\big)\right), &x \in
\Omega^{(i)}_{\varepsilon,2\ell_0,3\ell_0,\gamma }, \ i\in\{1,2,3\},
  \end{array}
\right.
\end{equation}
where the coefficients $\{w_{\alpha -1}^{(i)}\}_{i=1}^3$, $\{w_{0}^{(i)}\}_{i=1}^3$ and $\{w_{\alpha}^{(i)}\}_{i=1}^3$ form  classical solutions to the problems \eqref{limit_al-1}, \eqref{limit_0} and \eqref{limit_al}, respectively; the terms $\Pi_{\alpha-1}^{(i)},$
 $\Pi_{0}^{(i)},$ and $\Pi_{\alpha}$ are determined in \eqref{Pi_0};  and $N_{\alpha-1},$ $N_0,$ and $N_{\alpha}$ are solutions to the problem \eqref{N_p_alpha +k_prob}  for the corresponding values of the indices $p$ and $k.$

To obtain the estimate \eqref{main_U_1}, it is necessary to construct the partial sum $\mathfrak{U}_{2}^{(\varepsilon)}$ that additionally contains  the coefficients $\{w_{1}^{(i)}\}_{i=1}^3$,  $\{w_{\alpha+1}^{(i)}\}_{i=1}^3$, $\{u_{1}^{(i)}\}_{i=1}^3$ and $\{u_{\alpha +1}^{(i)}\}_{i=1}^3.$
This means that  $\{w_{0}^{(i)}\}_{i=1}^3,$ $\{w_{\alpha}^{(i)}\}_{i=1}^3$ and $\{w_{\alpha -1}^{(i)}\}_{i=1}^3$ must have  $C^2$ and $C^3$
smoothness, respectively, on the corresponding edges of the graph (see Remark~\ref{additional}).
To derive \eqref{main_U_2},  the partial sum $\mathfrak{U}_{3}^{(\varepsilon)}$ should be constructed and, as a result of which additional smoothness of the coefficients is required. Therefore, the following statements hold.

\begin{corollary}\label{Coro_3_1}
Let $\alpha \in (0,1)$ and,  in addition to the  assumptions made in Section~\ref{Sec:Statement}  the functions $\{\varphi^{(i)}\}_{i=1}^3$ belong to the smoothness class $C^4$ in their domains, $q_1 \in C^2([0,T]),$ and
$$
\partial_t \varphi^{(0)}\big|_{t=0}=\partial^2_{tt} \varphi^{(0)}\big|_{t=0}=0, \quad q''_1(0)=0.
$$

Then the inequality \eqref{main_U_1} hold.
\end{corollary}

\begin{corollary}\label{Coro_3_3}
 Let $\alpha \in (0,1)$ and, in addition to the  assumptions made in Section~\ref{Sec:Statement}, the functions $\{\varphi^{(i)}\}_{i=1}^3$ belong to the smoothness class $C^5$ in their domains, $\varphi^{(0)}$ belongs to $C^3$ in $t\in [0,T],$ $q_1 \in C^3([0,T]),$ $v^{(i)}_i \in C^4\big([0,\ell_i]\big),$ $i\in \{1, 2, 3\},$  and
 $$
\partial_t \varphi^{(0)}\big|_{t=0}=\partial^2_{tt} \varphi^{(0)}\big|_{t=0}= \partial^3_{ttt} \varphi^{(0)}\big|_{t=0}=0, \quad q''_1(0)=q'''_1(0)=0.
$$

Then the inequality \eqref{main_U_2} hold.
\end{corollary}

\section{Concluding remarks}

{\bf 1.} We have studied the influence of large boundary interactions $(\alpha < 1)$ on the asymptotic behaviour of the  solution $u_\varepsilon$  to the problem~\eqref{probl}. The constructed asymptotic expansion has revealed the dependence of the solution on the parameters $\varepsilon$ and $\alpha$ and other parameters of the problem (the geometric structure of the thin junction, including the local geometric irregularity of the node, through the constants $\{h_i\}$ and the values $\{\boldsymbol{d_{p\alpha +k}}(t)\}$ in the relations \eqref{trans*k-p*alpfa} and \eqref{const_d*k-p*alpfa}).
The principal  part of the Puiseux asymptotic expansion~\eqref{asymp_expansion} shows that the physical processes on the lateral surfaces of thin cylinders do cause cardinal changes in the global behavior of the solution (it becomes larger as the parameter $\varepsilon$ decreases).
From a physical point of view, it is advisable to consider the parameter $\alpha$ from the interval $(0,1),$  since for smaller values of the parameter $\alpha$ the solution becomes too large,  indicating the instability of the transport process in this case.

The approximation $\mathfrak{U}_{1}^{(\varepsilon)}$ constructed for the case $\alpha \in (0,1)$ indicates that
\begin{itemize}
  \item
  the principal coefficients  $\{\varepsilon^{\alpha -1} w_{\alpha -1}^{(i)}\}_{i=1}^3$ are directly affected by the boundary interactions $\{\varphi^{(i)}\}_{i=1}^3$ on the lateral surfaces of the thin cylinders through the solutions $\{u_{\alpha}^{(i)}\}_{i=1}^3,$ respectively;
  \item
  the coefficients $\{w_{0}^{(i)}\}_{i=1}^3$ take into account the inhomogeneous Dirichlet conditions on the bases of the thin cylinders;
  \item
  and only the coefficients  $\{\varepsilon^{\alpha} w_{\alpha}^{(i)}\}_{i=1}^3$ begin to feel the influence of the node boundary condition
  and physical processes inside the node through the value $\boldsymbol{d_{\alpha}}(t)$ (see \eqref{const_d_alpfa}), which depends  on both the interaction $\varphi^{(0)}$ on the node boundary and on the solution $N_{\alpha-1}.$
\end{itemize}

The node-layer solutions $N_{\alpha-1},$ $N_0,$ and $N_{\alpha}$ ensure the smoothness of the approximation $\mathfrak{U}_{1}^{(\varepsilon)}$ at the node, while the boundary-layer solutions $\Pi^{(i)}_{\alpha-1},$ $\Pi^{(i)}_0,$ and $\Pi^{(i)}_{\alpha}$ ensure the fulfillment of the boundary condition at the base of the thin cylinder $\Omega^{(i)}_{\varepsilon},$  $i\in \{2, 3\}.$

The estimates \eqref{main_2+} and \eqref{main_2}, proved in Theorem \ref{Th_1}, allow us to construct approximations of the solution with a given accuracy with respect to the small parameter $\varepsilon,$ which indicates the efficiency and usefulness of the proposed asymptotic approach.

\medskip

{\bf 2.} \ In the case when different intensities of boundary processes are observed in different parts of the boundary of a thin network, it is necessary to consider the corresponding intensity parameters. For example, $\alpha_1, \alpha_2, \alpha_3$ respectively  for the lateral surfaces of the thin cylinders $\Omega_\varepsilon^{(1)}, \Omega_\varepsilon^{(2)}, \Omega_\varepsilon^{(3)},$   and $\alpha_0$ for the node boundary. Then, for instance, the regular part of the asymptotics in  each thin cylinder $\Omega^{(i)}_\varepsilon$ will have the form
  \begin{equation}\label{new-form}
 \sum_{i=0}^{3}\sum\limits_{k=0}^{+\infty}  \varepsilon^{\alpha_i +k -1 } \, \left(w_{\alpha_i +k -1}^{(i)} (x_i) + u_{\alpha_i +k -1}^{(i)} \big( x_i, \tfrac{\overline{x}_i}{\varepsilon}, t \big) \right)
+
  \sum\limits_{k=0}^{+\infty}  \varepsilon^{k} \, \left(w_{k}^{(i)} (x_i) + u_{k}^{(i)} \big( x_i, \tfrac{\overline{x}_i}{\varepsilon}, t \big) \right).
\end{equation}

The  ansatz \eqref{new-form} shows that if $\alpha_0$ is less than $\alpha_1, \alpha_2, \alpha_3$, then activities at the node boundary can cause crucial changes in the entire transport process in a thin network.

\section*{Acknowledgments}
The authors  thank  for  funding by the  Deutsche Forschungsgemeinschaft (DFG, German Research Foundation) – Project Number 327154368 – SFB 1313.



\begin{thebibliography}{99}
\bibitem{Ang_2020}
{\sc M.  Anguiano}, {\em Existence, uniqueness and homogenization
of nonlinear parabolic problems with dynamical boundary conditions
in perforated media}, Mediterr. J. Math. 17, 18 (2020).

\bibitem{CarPanSir}
{\sc G. Cardone, G.P. Panasenko and Y. Sirakov}, {\em Asymptotic analysis and numerical modeling of mass transport in
tubular structures}, Mathematical Models and Methods in Applied Sciences 20(03) (2010) pp.  397-- 421.


\bibitem{Cio-Dam-Don-Gri-Zak2012}
{\sc D. Cioranescu,  A. Damlamian, P. Donato, G. Griso  and R. Zaki},
{\em The periodic unfolding method in domains with holes},
{SIAM Journal on Mathematical Analysis},
{44} (2012) pp. 718--760.

\bibitem{Che-Fri-Pia_1999}
{\sc G. Chechkin, A. Friedman  and A. Piatnitski},
{\em The boundary-value problem in domains with very rapidly oscillating boundary},
Journal of Mathematical Analysis and Applications, 231 (1999) pp.  213--234.


\bibitem{Con-Don_1988}
{\sc C. Conca and P. Donato}, {\em Non-homogeneous Neumann problems in domains with small holes},
M2AN - Mod\'elisation math\'ematique et analyse num\'erique,  22 (1988)  pp. 561-607.

\bibitem{Conca_1}
{\sc  C.~Conca, J.~Diaz,  and C.~Timofte}, {\em Effective chemical processes in porous media},
Mathematical Models and Methods in Applied Sciences, 13(10)  (2003) pp. 1437--1462

\bibitem{Conca}
{\sc  C.~Conca, J.~Diaz, A.~Linan, and C.~Timofte}, {\em Homogenization in
  chemical reactive flows}, Electron. J. Differential Equations. (2004) 1--22. http://eudml.org/doc/116613

\bibitem{Deu-Hoch_2004}
{\sc P. Deuflhard and  R. Hochmut}, {\em Multiscale analysis of thermoregulation in the human microvascular system},
Math. Meth. Appl. Sci. 27 (2004) pp. 971–989

\bibitem{Don-Gui-Oro_2018}
{\sc P. Donato, O. Guib\'e, A. Oropeza}, {\em Homogenization of quasilinear elliptic problems with nonlinear Robin conditions and $L^1$ data},
Journal de Math\'ematiques Pures et Appliqu\'ees, 120 (2018) pp.  91--129.


\bibitem{Gau_1994}
{\sc  A. Gaudiello}, {\em Asymptotic behaviour of non-homogeneous Neumann problems in domains with oscillating boundary}, Ric. Mat. 43 (2) (1994) pp. 239--292.


\bibitem{Gau-Mel_2019}
{\sc  A. Gaudiello and T.~Mel'nyk}, {\em Homogenization of a nonlinear monotone problem with a big
nonlinear Signorini boundary interaction in a domain with highly rough boundary}, Nonlinearity, 32 (2019) pp.~5150--5169.

\bibitem{Gom-Lob-Per-San_2018}
{\sc D. G\'omez, M. Lobo, E. P\'erez and  E. Sanchez-Palencia}, {\em Homogenization
in perforated domains: a Stokes grill and an adsorption process}, Applicable Analysis, 97 (2018)
pp. 2893--2919.

\bibitem{Gom-Lob-Per-Pod-Sha_2018}
{\sc D. G\'omez, M. Lobo, E. P\'erez,  A.  Podolskii and T.  Shaposhnikova}, {\em Unilateral problems for the p-Laplace operator in perforated media involving large parameters},  ESAIM: COCV 24 (2018), pp. 921--964.



\bibitem{Hor-Jag_1991}
{\sc U. Hornung  and W. J\"ager}, {\em  Diffusion, convection, adsorption and reaction of chemicals in porous media},  Journal of
Differential Equations, 92 (1991) pp. 199–225.

\bibitem{Mel-Kle_2019}
{\sc A.~Klevtsovskiy and T.~Mel'nyk}, {\em Influence of the node on the
  asymptotic behaviour of the solution to a semilinear parabolic problem in a
  thin graph-like junction}, Asymptotic Analysis, 113 (2019), pp.~87--121.



\bibitem{Lad_Sol_Ura_1968}
{\sc O.~Ladyzhenskaya, V.~Solonnikov, and N.~Ural’tseva}, {\em Linear and
  Quasilinear Equations of Parabolic Type}, vol.~23, AMS, Transl. Math.
  Monographs, 1968.

\bibitem{Mel_1991}
{\sc T.~Mel'nyk}, {\em Averaging of elliptic equations that describe processes in strongly inhomogeneous thin punctured domains with rapidly changing thickness.}  Dokl. Akad. Nauk Ukrain. SSR 1991, no. 10, 15–18,

\bibitem{Mel_IJB-2019}
{\sc T.~Mel'nyk}, {\em Asymptotic analysis of a mathematical model of the
  atherosclerosis development}, International Journal of Biomathematics, 12
  (2019), p.~1950014.

\bibitem{M-AA-2021}
{\sc T.~Mel'nyk}, {\em Asymptotic approximations for eigenvalues and
  eigenfunctions of a spectral problem in a thin graph-like junction with a
  concentrated mass in the node}, Analysis and Applications, 19 (2021),
  pp.~875--939.

\bibitem{Mel-Pop_2009}
{\sc T. Mel'nyk and  A. Popov}
{\em Asymptotic approximations of solutions to parabolic boundary value problems
in thin perforated domains of rapidly varying thickness},
 J Math Sci,  162 (2009) pp. 348–372.


\bibitem{Mel-Siv_2011}
{\sc T. Mel'nyk and O. Sivak}, {\em
Asymptotic approximations for solutions to
quasilinear and linear parabolic problems
with different perturbed boundary conditions in
perforated domains}, J Math Sci 177 (2011) 50--70.


 \bibitem{Mel-Sad_2019}
 {\sc T. Mel'nyk and  D.  Sadovyj}, {\em  Multiple-Scale Analysis of Boundary-Value Problems in
Thick Multi-Level Junctions of Type 3:2:2}, Springer, 2019.


\bibitem{Mel-Roh_preprint-2022}
{\sc T.~Mel'nyk and C.~Rohde}, {\em Asymptotic expansion for
  convection-dominated transport in a thin graph-like junction}, E-print:
  arXiv:2208.05812,  (2022),  https://arxiv.org/abs/2208.05812

\bibitem{Mel-Roh_preprint-2023}
{\sc T.~Mel'nyk and C.~Rohde}, {\em Asymptotic approximations for semilinear parabolic convection-dominated transport problems\\ in thin graph-like networks},  J. Math. Anal. Appl.  (2023) in print;  see also E-preprint:  arXiv:2302.10105v1
https://doi.org/10.48550/arXiv.2302.10105

\end{thebibliography}

\end{document}